\documentclass[12pt]{amsart}
\usepackage{enumerate}
\usepackage{amsthm}
\usepackage{amsmath,amssymb,latexsym,amscd}
\usepackage{tikz}
\usepackage[colorlinks=true, urlcolor=Blue, pdfborder={0 0 0}]{hyperref}
\makeatletter
\@namedef{subjclassname@2010}{%
	\textup{2020} Mathematics Subject Classification}
\makeatother
\usepackage[noadjust]{cite}
\theoremstyle{definition}
\newtheorem{theorem}{Theorem}[section]

\newtheorem{proposition}[theorem]{Proposition}

\theoremstyle{definition}
\theoremstyle{definition}

\newtheorem{example}[theorem]{Example}
\usepackage{indentfirst}

\begin{document}
	\baselineskip=17pt
	\title[]{Free Circle Actions on The Product of Three Spheres}
	\author[Hemant Kumar Singh and Dimpi]{Hemant Kumar Singh and Dimpi}
		\address{  Hemant Kumar Singh\newline\indent 
		Department of Mathematics\newline\indent University of Delhi\newline\indent 
		Delhi -- 110007, India.}
	\email{hemantksingh@maths.du.ac.in}
	\address{ Dimpi \newline 
		\indent Department of Mathematics\indent \newline\indent University of Delhi\newline\indent 
		Delhi -- 110007, India.}
	\email{dimpipaul2@gmail.com}
	
	\date{}
	\thanks{The first author of the paper is  supported by Faculty Research Program of the IoE Scheme of the University of Delhi with reference no. Ref.No./IoE/2024-2025/12/FRP}
	\begin{abstract} 
		
		\noindent The orbit spaces of free $\mathbb{S}^0$-actions on the mod 2 cohomology product of three spheres, $\mathbb{S}^n\times \mathbb{S}^m \times \mathbb{S}^l, 1 \leq n \leq m \leq l,$ have been determined in \cite{Dimpi}.  In this paper, we extend these findings  to free $\mathbb{S}^1$-actions on the rational cohomology product of three spheres. This extension also builds upon the work of Dotzel et al. \cite{Dotzel}, who studied free circle actions on the rational cohomology  product of two spheres. Additionally, we establish Borsuk-Ulam type theorems. 
		
	\end{abstract}
	

	\maketitle
	\section {Introduction}
	Let $G$ be a compact Lie group acting on a finitistic space $X.$ The determination of the orbit space $X/G,$ particularly when $G$ acts freely on $X,$ has drawn widespread interest from mathematicians.  Firstly, in $1925$-$26,$ H. Hopf raised the question to determine the orbit spaces $\mathbb{S}^n/G,$ where $G$ is a finite cyclic group. Later, in 1964, Hirsch et al. established that the orbit space $\mathbb{S}^n/\mathbb{Z}_2$ is homotopy type of $\mathbb{R}P^n,$ where $\mathbb{Z}_2$ acts freely on $\mathbb{S}^n.$  Further, the orbit spaces of finite group actions on $n$-sphere $\mathbb{S}^n$ have been studied in \cite{livesay,rice,rubin,Ritter}.  However, very little is known if the space is a compact manifold other than a sphere. Tao \cite{tao} determined
	the orbit spaces of free involutions on $\mathbb{S}^1 \times  \mathbb{S}^2$. Ritter \cite{Ritter} extended the Tao’s results
	to cyclic groups of order $2n.$ In 1972, Ozeki et al. \cite{ozeki} determined the orbit space of a
	free circle action on a manifold having cohomology of product of spheres $\mathbb{S}^{2n+1}\times \mathbb{S}^{2n+1}$ with integer coefficients. The orbit spaces
	of free actions of $G=\mathbb{Z}_p, p$ a prime, or $G=\mathbb{S}^d,d=1$ or $3,$ on a finitistic space $X$ having mod $p$ or rational cohomology of  product of two spheres $\mathbb{S}^n \times \mathbb{S}^m,1\leq n\leq m,$ have been studied in \cite{Dotzel} and \cite{anju}. Dey et al. \cite{dey}, Morita et al. \cite{morita}, and  Singh \cite{singh} discussed the orbit spaces of free involutions on real Milnor manifolds, Dold manifolds,
	and the product of two projective spaces, respectively. Recently, in \cite{Dimpi}, we have discussed the orbit space  of free involutions on a finitistic space having mod $2$ cohomology of the product of three spheres $\mathbb{S}^n \times  \mathbb{S}^m \times \mathbb{S}^l,$ $1\leq n  \leq m \leq l .$ 
	
		Expanding upon this area of research, this paper investigates the orbit spaces  of free circle actions on a finitistic space having rational  cohomology of the product of three spheres $\mathbb{S}^n \times  \mathbb{S}^m \times \mathbb{S}^l,$ $1\leq n  \leq m \leq l .$ For  instance, consider  the complex Stiefel manifold $V_{n,n-3}$ whose integral cohomology is isomorphic  to the product of three spheres $\mathbb{S}^{2n-9}\times \mathbb{S}^{2n-7}\times\mathbb{S}^{2n-5},$ for all $n\geq 5$ \cite{adam}. This complex Stiefel manifold $V_{n,n-3}$  admits a free circle action  defined by $(v_1,v_2,\cdots, v_{n-3})\mapsto (\lambda v_1,\lambda v_2,\cdots, \lambda v_{n-3}), $ where $\lambda \in \mathbb{S}^1$ and $v_i's, 1\leq i \leq n-3,$ are orthonormal vectors in $\mathbb{C}^n.$ If we consider diagonal actions of $G=\mathbb{S}^1$ on $SU(3)\times \mathbb{S}^{2l+1}$ and $U(2)\times \mathbb{S}^{2l+1},$ where $G=\mathbb{S}^1$ acts on $\mathbb{S}^{2l+1}$ defined by $(z_1,z_2,\cdots, z_{l+1})\mapsto (\lambda z_1,\lambda z_2,\cdots, \lambda z_{l+1}), $ where $\lambda \in \mathbb{S}^1,$  while acting trivially on $SU(3)$ and $U(2).$ This gives  free circle action  on  $SU(3)\times \mathbb{S}^{2l+1}$ and $U(2)\times \mathbb{S}^{2l+1},$  having integral cohomologies isomorphic to $\mathbb{S}^3 \times  \mathbb{S}^5 \times \mathbb{S}^{2l+1}$ and $\mathbb{S}^1 \times  \mathbb{S}^3 \times \mathbb{S}^{2l+1},$ respectively \cite{borel}. 	As an application, we have also determined the  Borsuk-Ulam type results.
			\section{Preliminaries}  
			\noindent	In this section, we recall some basic key results that are used throughout the paper.
		Let $G$ be a compact Lie group  acting on a finitistic space $X.$  Then there exists a universal principal $G$-bundle $G\hookrightarrow E_G\rightarrow B_G.$ If $X$ is free $G$-space then  the associated Borel fibration is $ X \stackrel{i} \hookrightarrow X_G \stackrel{\pi} \rightarrow B_G,$ where $X_G = (X\times E_G)/G$ (Borel space) obtained by diagonal action of $G$ on space $X\times E_G.$  We recall some results of the Leray-Serre spectral sequence associated with Borel fibration $ X \stackrel{i} \hookrightarrow X_G \stackrel{\pi} \rightarrow B_G.$ For proofs, we refer \cite{bredon,mac}.
			
			\begin{proposition}(\cite{mac})
				Suppose that  the system of local coefficients on $B_G$ is simple.	Then the  homomorphisms $i^*: H^*(X_G) \rightarrow H^*(X)$ and $\pi^*: H^*(B_G) \rightarrow H^*(X_G)$ are the edge homomorphisms, \begin{center}
					$ H^k(B_G)=E_2^{k,0}\rightarrow E_3^{k,0}\rightarrow \cdots E_k^{k,0}\rightarrow E_{k+1}^{k,0} = E_{\infty}^{k,0} \subset H^k(X_G),$ and $  H^i(X_G) \rightarrow E_{\infty}^{0,i} \hookrightarrow E_{l+1}^{0,i} \hookrightarrow E_{l}^{0,i} \hookrightarrow \cdots \hookrightarrow E_{3}^{0,i} \hookrightarrow E_2^{0,i} \hookrightarrow H^i(X),$ respectively.
				\end{center} 
			\end{proposition}
			
			\begin{proposition}(\cite{mac})\label{2.4}
				Let $G=\mathbb{S}^1$ act on a finitistic space $X$ and $\pi_1(B _G)$ acts trivially on $H^*(X).$ Then the system of local coefficients on $B_G$ is simple and $$E^{k,i}_2= H^{k}(B_G) \otimes  H^i(X), ~k,i \geq 0.$$
			\end{proposition}

			\begin{proposition}(\cite{mac})\label{mac}
				Let $G=\mathbb{S}^1$   act freely on a finitistic space $X.$ Then the Borel space $X_G$ is homotopy equivalent to the orbit space $X/G.$
			\end{proposition}
			\begin{proposition}(\cite{bredon})\label{prop 4.5}
				Let $G=\mathbb{S}^1$   act freely on a finitistic space $X.$ If $H^i(X;\mathbb{Q})=0~ \forall~ i>n,$ then $H^i(X/G;\mathbb{Q})=0~ \forall~ i>n.$ 
			\end{proposition}

			\noindent Recall  that 
			$H^*(\mathbb{S}^n \times \mathbb{S}^m \times \mathbb{S}^l; \mathbb{Q})=\mathbb{Q}[a,b,c]/<{a^{2},b^2, c^2}>,$ where deg $a=n,$ deg $b=m$ and  deg $c=l, ~1 \leq n \leq m\leq l.$\\
			\noindent Throughout the paper,  
			$H^*(X)$ will denote the \v{C}ech cohomology of a space $X$ with coefficient group $G=\mathbb{Q},$ and   $X\sim_\mathbb{Q} Y,$  means $H^*(X;\mathbb{Q} )\cong H^*(Y;\mathbb{Q})$ as graded cohomology algebras.

	\section{The  Cohomology  Algebra of  The Orbit Space of  Free Circle Actions on $ \mathbb{S}^n \times \mathbb{S}^m \times \mathbb{S}^l $}
	\noindent In this section, we determine the cohomology ring of orbit space of free $G = \mathbb{S}^1$ actions on finitistic spaces $X$ having rational cohomology of the product of three spheres, that is, $X\sim_{\mathbb{Q}} \mathbb{S}^n \times \mathbb{S}^m \times \mathbb{S}^l, $ where $1 \leq n \leq  m\leq l.$ \\
	 Let $\{E_r^{*,*},d_r\}$ be the Leray-Serre spectral sequence of the Borel fibration $X \hookrightarrow X_G \rightarrow B_G.$ As $G $ act freely on a finitistic space $X\sim_{\mathbb{Q}} \mathbb{S}^n \times \mathbb{S}^m \times \mathbb{S}^l, $ we get the spectral sequence $\{E_r^{*,*},d_r\}$ must be nondegenerate, that is, $E_2^{*,*}\neq E^{*,*}_\infty.$ Also, as $\pi_1(B_G)$ is  trivial,  by Proposition \ref{2.4}, we get  $$E_2^{k,i}=H^k(B_G) \otimes H^i(X)~\forall~  k,i\geq 0.$$ Thus $E_2^{k,i}\cong \mathbb{Q},$ for $k=0,n,m,l,n+m,n+l,m+l,n+m+l,$ and $\forall~ i\geq 0.$ \\  Let $a,b,c$ be generators of $H^*(X).$  The non-triviality of  differentials $d_r(1\otimes a),d_r(1\otimes b)$ or $d_r(1\otimes c)$ occurs under the following conditions: \begin{itemize}
		\item if $d_{r_1}(1\otimes a) \neq 0,$ then $r_1=n+1$ \& $n$ is odd, 
		\item if $d_{r_2}(1\otimes b) \neq 0,$ then $r_2=m-n+1$ \& $m-n$ is odd or $r_2=m+1$ \& $m$ is odd, and 
		\item if $d_{r_3}(1\otimes c) \neq 0,$ then $r_3=l-m-n+1$ \& $l-m-n$ is odd,  $r_3=l-m+1$ \& $l-m$ is odd, $r_3=l-n+1$ \& $l-n$ is odd or $r_3=l+1$ \& $l$ is odd. 
	\end{itemize} 

We have proved the following theorems under  conditions  that at least one  differential $d_r(1\otimes a),d_r(1\otimes b)$ or $d_r(1\otimes c)$  must be   nonzero.

\begin{theorem}\label{thm 3.6}
		Let $G=\mathbb{S}^1$ act freely on a finitistic space $X \sim_{\mathbb{Q}} \mathbb{S}^n \times \mathbb{S}^m \times \mathbb{S}^l, $ where $1\leq n\leq m \leq l.$ If $d_{r_1}(1\otimes a)\neq 0,$  then $H^*(X/G)$ is isomorphic to one of the following graded commutative algebras:

	\begin{enumerate}
	\item  $\mathbb{Q}[x,y,w,z]/<x^{\frac{n+1}{2}}, I_j, z^2,yz,wz>_{1\leq j \leq 3}$,\\ where $n$ odd,	deg $x=2,$ deg $y=m$, deg $w=l~\&$ deg $z=m+l;$ 
	$I_1=y^2+a_1z+a_2x^{\frac{2m-l}{2}}w,I_2=w^2+a_3x^{\frac{l-m}{2}}z,$ and $ I_3=yw+a_4z,$  $a_i \in \mathbb{Q}$, $1\leq i \leq 4;$  $a_1=0$ if $n \leq m < l,$  $a_2=0$ if $2m > n+l-1 $ or $l$ odd, $a_3=0$ if $l-m+1 > n$ or $l-m$ odd.

	\item $\mathbb{Q}[x,y,w,z]/<x^{\frac{n+1}{2}}, I_j, w^2,z^2,wz,yz,x^{\frac{l-m+1}{2}}y,x^{\frac{l-m+1}{2}}w>_{1\leq j \leq 2}$, \\ where $n$ odd, deg $x=2,$ deg $y=m,$ deg $w=n+m~\&$ deg $z=m+l;$
	$I_1=y^2+a_1x^{\frac{m-n}{2}}w,$  and
	$I_2=yw+a_2x^{\frac{m+n-l}{2}}z,$	$a_i \in \mathbb{Q},$ $1\leq i \leq 2;$ $a_1=0$ if $2m>n+l-1.$

\end{enumerate}
		
	\end{theorem}
	\begin{proof}
		As $d_{r_1}(1\otimes a)\neq 0,$ we get $r_1=n+1,$ where $n$ is odd. So, $d_{n+1}(1\otimes a)=c_0t^{\frac{n+1}{2}}\otimes 1,0\neq c_0 \in \mathbb{Q}.$ For the remaining differentials, the following four cases are possible:  (i) $d_{r_2}(1\otimes b)=d_{r_3}(1\otimes c)= 0,$ (ii) $d_{r_2}(1\otimes b) \neq 0 ~\&~ d_{r_3}(1\otimes c)=0,$ (iii) $d_{r_2}(1\otimes b)= 0 ~\&~ d_{r_3}(1\otimes c) \neq 0,$ and (iv) $ d_{r_2}(1\otimes b)\neq 0 ~ \& ~ d_{r_3}(1\otimes c) \neq 0 .$
		
		\noindent {\bf Case (i):} $d_{r_2}(1\otimes b)=0~\&$ $d_{r_3}(1\otimes c)=0.$\\
		In this case, we get  $d_{n+1}(1\otimes ab)=c_0t^{\frac{n+1}{2}}\otimes b $, $d_{n+1}(1\otimes ac)=c_0t^{\frac{n+1}{2}}\otimes c$ and $d_{n+1}(1\otimes abc)=c_0t^{\frac{n+1}{2}}\otimes bc,$ where $0\neq c_0\in \mathbb{Q}.$ So, $d_r=0$ for all $r>n+1.$ Thus  $E_{n+2}^{*,*}=E_{\infty}^{*,*}.$ \\ If $n\leq m<l,$ then  $E_{\infty}^{p,q} \cong \mathbb{Q}$ for $0 \leq p \leq n-1,p$ even $\&$ $q=0,m,l$ or $m+l,$ and  zero otherwise. \\ 
		For $l\geq m+n,$ the  cohomology groups of 
		$X/G$ are as follows:
		$$
		H^k(X_G)=
		\begin{cases}
			\mathbb{Q}  & j \leq k \leq n+j-1, j=0,m,l,m+l; k-j \mbox{~even},\\
				0 & \mbox{otherwise}
		
		\end{cases}
		$$ and, for $l< m+n,$ we have    
		$$
		H^k(X_G)=
		\begin{cases}
			\mathbb{Q}  & j \leq k < n+j-1, j=0,m+l; k-j \mbox{ even},\\
				\mathbb{Q}  & m \leq k < l;  k-m \mbox{~even}, m+n-1 < k \leq n+l-1;  k-{l} \mbox{~even},\\
			\mathbb{Q} \oplus \mathbb{Q} & l \leq k \leq m+n-1,  k-m ~\&~ k-l \mbox{~even}, \\
			0 & \mbox{otherwise.}
		\end{cases}
		$$ 
		The permanent cocycles  $t\otimes 1, 1 \otimes b,$ $1 \otimes c$ and $1 \otimes bc$ of $E_2^{*,*}$ determine the elements  $x\in E_{\infty}^{2,0}, u \in E_{\infty}^{0,m},$ $v \in E_{\infty}^{0,l}$ 
		and $s \in E_{\infty}^{0,m+l},$ respectively.  Thus the total complex is given by \begin{center}
			Tot $E_{\infty}^{*,*} = \mathbb{Q}[x,u,v,s]/<x^{\frac{n+1}{2}},u^2+\gamma_1v,v^2,s^2,uv+\gamma_2s,us,vs>,$
		\end{center}where deg $x=2,$ deg $u= m,$ deg $v= l~\&$   deg $s= m+l$ and ${\gamma_1,\gamma_2 \in \mathbb{Q}},$ $\gamma_1=0$ if $l \neq 2m.$ Let $y \in H^m(X_G), w \in H^l(X_G)$ and $ z \in H^{m+l}(X_G)$ such that $i^*(y)=b,i^*(w)=c,$ and $i^*(z)=bc,$ respectively. Clearly,  $I_1= y^2+b_1x^{\frac{2m-l}{2}}w=0,$ $I_2= w^2+b_2x^{\frac{l-m}{2}}z=0 ~\&$ $I_3= yw+b_3z=0,$   where $b_i \in \mathbb{Q}, 1\leq i \leq 3,$ $b_1=0 $ if $2m>n+l-1$ or $l$ is odd  and $b_2=0$ if $l+1>m+n$ or $l-m$ is odd.   By Proposition \ref{mac}, the cohomology ring of the orbit space $X/G$ is given by $$ \mathbb{Q}[x,y,w,z]/<x^{\frac{n+1}{2}}, I_j, z^2,yz,wz>_{1\leq j \leq 3}, $$  where deg $x=2,$ deg $y=m$, deg $w=n+m$ and deg $z=m+l.$ This realizes possibility (1).
	For  $n\leq m=l,$  we again get possibility (1). In this case, the $E_{\infty}$ page of the spectral sequence  is as follows: $E_{\infty}^{p,q} \cong \mathbb{Q}$ for $0 \leq p \leq n-1, p$ even $\&$ $q=0$ or $2m ;$  $E_{\infty}^{p,q}\cong \mathbb{Q} \oplus \mathbb{Q}   $ for $0 \leq p \leq n-1, p$ even $\&$ $q=m, $ and   zero otherwise. \\ 
	{\bf Case (ii)}  $d_{r_2}(1 \otimes b)\neq 0 ~~\&~~ d_{r_3}(1 \otimes c)=0.$\\ This case is possible only when $n=m.$ We have $d_{n+1}(1 \otimes a)=c_0t^{\frac{n+1}{2}}\otimes 1 ~\& ~d_{n+1}(1\otimes b)= c_1t^{\frac{n+1}{2}}\otimes 1, 0\neq c_1\in \mathbb{Q}.$ Consequently,  $d_{n+1}(1 \otimes ac)= c_0t^{\frac{n+1}{2}}\otimes c~\&~ d_{n+1}(1\otimes bc)= c_1t^{\frac{n+1}{2}}\otimes c,$ $d_{n+1}(1 \otimes ab)= t^{\frac{n+1}{2}}\otimes (c_1a-c_0b)$ ~$\&$~ $d_{n+1}(1 \otimes abc)=t^{\frac{n+1}{2}}\otimes (c_1ac-c_0bc).$ Thus $E_{n+2}^{*,*}=E_{\infty}^{*,*}.$ The elements  $t\otimes 1,1\otimes (c_1a-c_0b), 1 \otimes c$ and $1 \otimes c(c_1a-c_0b)$ are   permanent cocycles, and the  cohomology groups and cohomology algebra of $X/G$ are  same as in the case (i) when $n=m\leq l.$  \\
		{\bf Case (iii)}  $d_{r_2}(1 \otimes b)=0 ~~\&~~ d_{r_3}(1 \otimes c)\neq 0.$\\ As $d_{r_1}(1 \otimes a)\neq 0 ~~\&~~ d_{r_3}(1 \otimes c)\neq 0,$  this case is not possible when $l> m+n$ or $n<m=l.$ \\
		Now, consider $n \leq m <l\leq m+n.$ If $l<n+m$ then we must have  $d_{l-m+1}$  nontrivial. 
		 So, we get $d_{l-m+1}(1\otimes c)=c_2t^{\frac{l-m+1}{2}}\otimes b$ and $d_{l-m+1}(1\otimes ac)=c_2t^{\frac{l-m+1}{2}}\otimes ab.$ Clearly, $d_r=0$  for $l-m+1<r<n+1.$ Since $d_{n+1}(1 \otimes a)=c_0t^{n+1}\otimes 1,$ we get  $d_{n+1}(1 \otimes abc)=c_0t^{n+1}\otimes bc.$ Thus $E_{n+2}^{*,*}=E_{\infty}^{*,*},$  and $E_{\infty}^{p,q} \cong \mathbb{Q}$ for $0 \leq p \leq n-1$, $p$ even $\&$ $q=0$ or $m+l; E_{\infty}^{p,q} \cong \mathbb{Q}$ for $0 \leq p \leq l-m-1$, $p$ even $\&$ $q=m$ or $n+m,$ and zero  otherwise. Thus the cohomology groups of $X/G$ are as follows 
		$$
		H^k(X_G)=
		\begin{cases}
			\mathbb{Q}  & j \leq k \leq n+j-1, j=0,m+l; k-j \mbox{ even}\\
			\mathbb{Q}  &  m+j \leq k \leq l+j-1, j=0,n;k-(m+j)\mbox{ even}\\
			0 & \mbox{otherwise.}
		\end{cases}
		$$ 
		The permanent cocycles  $t\otimes 1, 1 \otimes b,$ $1 \otimes ab$ and $1 \otimes bc$ of $E_2^{*,*}$ determine the elements $x\in E_{\infty}^{2,0}, u \in E_{\infty}^{0,m},$ $v \in E_{\infty}^{0,n+m}$ 
		$\& ~s \in E_{\infty}^{0,m+l},$ respectively. 
		The total complex is given by \begin{center}
			Tot $E_{\infty}^{*,*} = \mathbb{Q}[x,u,v,s]/<x^{\frac{n+1}{2}},u^2+\gamma_1v,v^2,s^2,uv,us,vs,x^{\frac{l-m+1}{2}}u,x^{\frac{l-m+1}{2}}v>,$
		\end{center} where  deg $x=2,$ deg $u$ = $m$, deg $v$ = $n+m~\&$  deg $s= m+l,$ $ \gamma_1 \in \mathbb{Q},\gamma_1=0$ if $n< m.$ Let $ y \in H^m(X_G), w \in H^{n+m}(X_G)$ and $ z \in H^{m+l}(X_G)$ such that $ i^*(y)=b,i^*(w)=ab,$ and $i^*(z)=bc,$ respectively.  Thus for  $n\leq m<l<n+m,$ we have 	$I_1=y^2+b_1x^{\frac{m-n}{2}}w=0,$ $I_2=yw+b_2x^{\frac{m+n-l}{2}}z=0,$  where $b_1,b_2\in \mathbb{Q},$   $b_1=0 $ if $2m>n+l-1.$
		Thus the  cohomology ring  of the orbit space $X/G$ is given by  $$ \mathbb{Q}[x,y,w,z]/<x^{\frac{n+1}{2}}, I_j, w^2,z^2,wz,yz,x^{\frac{l-m+1}{2}}y,x^{\frac{l-m+1}{2}}w>_{1\leq j \leq 2},$$  where 
		deg $x=1,$ deg $y=m$, deg $w=n+m,$ and deg $z=m+l.$   This realizes possibility (2).\\ Now, if $n\leq m<l=m+n,$ then we must have  $d_{n+1}(1\otimes ab)\neq 0$ and $d_{n+1}(1 \otimes c)\neq 0.$ It is easy to see that  the cohomology algebra of $X/G$ is same as in the case (i) when $n \leq m<l$ and $l=m+n.$\\ Finally, we consider $n=m=l.$ We must have $d_{n+1}(1 \otimes c)=c_2t^{\frac{n+1}{2}}\otimes 1.$ So, we get $E_{n+2}^{*,*}=E_{\infty}^{*,*},$  and the cohomology algebra is  same as in the case (i) when $n=m=l.$ \\
		{\bf Case (iv)} $d_{r_2}(1 \otimes b)\neq 0 ~~\&~~ d_{r_3}(1 \otimes c)\neq 0.$\\ In this case, we must have $n=m \leq l.$ First, we consider $n=m<l.$ Then  we must have $l=2n.$    We get $d_{n+1}(1\otimes a)=c_0t^{\frac{n+1}{2}}\otimes 1,d_{n+1}(1 \otimes b)=c_1t^{\frac{n+1}{2}}\otimes 1, d_{n+1}(1\otimes ab)=t^{n+1}\otimes (c_1a+c_0b) ~\& ~ d_{n+1}(1 \otimes c)=c_2t^{n+1}\otimes (c_3a+c_4b),0 \neq c_i\in \mathbb{Q}, 0\leq i \leq 4.$ Consequently, $E_{n+2}^{*,*}=E_{\infty}^{*,*}.$ The elements  $t\otimes 1 , 1\otimes (c_1a+c_0b), 1 \otimes (c_2ab+c_1c)$ and $1 \otimes (ac-bc)$ are permanent cocycles, and the cohomology algebra of $X/G$ is  same as in the  case (i) when $n=m < l=2n.$ \\   Next, we consider $n=m=l.$ We get $d_{n+1}(1\otimes a),d_{n+1}(1 \otimes b)$ and $ d_{n+1}(1 \otimes c)$ are nonzero. So, $E_{n+2}^{*,*}=E_{\infty}^{*,*},$  and
		the cohomology algebra of $X/G$ is same as in the case (i) when $n=m=l.$  
	\end{proof}
\begin{theorem} \label{thm 3.7}
	Let $G=\mathbb{S}^1$ act freely on a finitistic space $X \sim_\mathbb{Q} \mathbb{S}^n \times \mathbb{S}^m \times \mathbb{S}^l, $ where $1\leq n\leq m \leq l.$ If $d_{r_1}(1\otimes a)=0 $ and $d_{r_2}(1\otimes b) \neq 0,$ then $H^*(X/G)$ is isomorphic to one of the following graded commutative algebras:
	\begin{enumerate}
	
	\item  $\mathbb{Q}[x,y,w,z]/<Q(x), I_j, z^2,x^{\frac{m-n+1}{2}}y,x^{\frac{m-n+1}{2}}z>_{1\leq j \leq 4},$ \\ where deg $x=2,$ deg $y=n,$ deg $w=l~\&$ deg $z=n+l;$
	$I_1=y^2+a_1x^{n}+a_2x^{\frac{2n-l}{2}}w+a_3x^{\frac{n}{2}}y,I_2=w^2+a_4x^{\frac{l}{2}}w+a_5x^{\frac{l-n}{2}}z+a_6x^{l}, I_3=yw+a_7z+a_8x^{\frac{n}{2}}w+a_9x^{\frac{n+l}{2}} ~\&~ I_4 = yz+a_{10}x^{\frac{n}{2}}z+a_{11}x^{n}w+a_{12}x^{\frac{2n+l}{2}}$ $a_i \in \mathbb{Q},$ $1\leq i \leq 12;$  $a_2=0$ if $l$ odd or $l>2n,$ $a_3=0$ if $n$ odd or  $m-1<2n, a_{10}=0$ if $2n+1>m;$   $Q(x)=x^{\frac{n+m+j'+1}{2}},$ $j'=0 ~\&$ $l.$   \\
	If $j'=0,$ then  $ a_8=0$ if $n$ odd  and  $a_i=0$ for $i=6,9,12.$\\
	If $j'=l,$ then $a_9=0$  if $n$ odd and    $a_{4}=a_{5}=a_{8}=a_{11}=0.$ 
	\item  $\mathbb{Q}[x,y,w,z]/<x^{\frac{m+1}{2}}, I_j, z^2,wz>_{1\leq j \leq 4},$\\ where $m$ odd, 	deg $x=2,$ deg $y=n$, deg $w=l~\&$ deg $z=n+l;$
	$I_1=y^2+a_1x^{\frac{n}{2}}y+a_2x^{\frac{2n-l}{2}}w+a_3x^{n}+a_4z,I_2=w^2+a_5x^{\frac{l-n}{2}}z, I_3=yw+a_6z+a_7x^{\frac{n}{2}}w~\&~ I_4 = yz+a_{8}x^{\frac{n}{2}}z+a_{9}x^{n}w,$ $a_i \in \mathbb{Q}$, $1\leq i \leq 9;$  $a_2=0$ if $m+l-1 <2n,$ $a_3=0$ if $n\neq  l ,$ $a_7=a_8=0$ if $n$ odd,  $a_5=0$ if $ m+n-1< l$ and $a_9=0$ if  $2n+1>m.$ 
	
	\item  $\mathbb{Q}[x,y,w,z]/<Q(x), I_j, z^2,wz,x^{\frac{m-n+1}{2}}y,x^{\frac{m-n+1}{2}}z,a_0x^{\frac{l-m-n+1}{2}}w>_{1\leq j \leq 4},$\\
	where deg $x=2,$ deg $y=n$, deg $w=n+m~\&$ deg $z=n+l;$ $I_1=y^2+a_1x^{n}+a_2x^{\frac{n}{2}}y,I_2=w^2+a_3x^{\frac{n+m}{2}}w+a_4x^{n+m}+a_5x^{\frac{2m+n-l}{2}}z, I_3= yw+a_6x^{\frac{n}{2}}w+a_7x^{\frac{2n+m}{2}}+a_8x^{\frac{n+m-l}{2}}z~\&~ I_4 = yz+a_9x^{\frac{n+l-m}{2}}w+a_{10}x^{\frac{n}{2}}z+a_{11}x^{\frac{2n+l}{2}},$ $a_i\in \mathbb{Q}, 0 \leq i \leq 11;$ $ a_2=0$ if $n$ odd or $m-1<2n;a_3=0$ if $l-1<n+m$ or $m=l;a_{10}=0$ if $m<2n+1$ or $n$ odd,   $Q(x)=x^{\frac{l+j'+1}{2}}, j'=0 ~\mbox{or}~ n+m.$ \\ If $j'=0,$ then $a_1=0$ if $l-1<2n,$  $a_4=0$ if $l-1<2n+2m$ or $m=l,a_5=0$ if $l-1<m+2n$ or  $m=l,a_6=0$ if $n$ odd, $a_7=0$ if $l-1<2n+m,$ $a_8=0$ if  $l <2n+1 ,$  $a_{10}=0$ if $m<2n+1~\&~a_0=a_{11}=0.$\\
	If $j'=n+m ,$ then  $a_6=0$ if $l<2n+m+1,$ and $a_5=a_8=a_9=0~ \&~ a_0 =1.$ 	  
	\item 	$\mathbb{Q}[x,y,w,z]/<x^{\frac{m+1}{2}}, I_j,w^2,
	z^2,wz,x^{\frac{l-n+1}{2}}y,x^{\frac{l-n+1}{2}}w>_{1\leq j \leq 3},$\\ where $m$ odd, deg $x=2,$ deg $y=n$, deg $w=n+m~\&$ deg $z=n+l;$ 
	$I_1=y^2+a_1x^{n}+a_2x^{\frac{n}{2}}y +a_3w,I_2=yw +a_4x^{\frac{n}{2}}w+a_5x^{\frac{m+n-l}{2}}z~\&~ I_3= yz+a_6x^{\frac{n}{2}}z+a_7x^{\frac{l+n-m}{2}}w,$ $a_i\in \mathbb{Q}, 1 \leq i \leq 7;$  $a_1=a_7=0$ if $m-1<2n,$ $ a_2=a_4=0$ if  $l<2n+1,$ $a_3=0$ if $n < m.$
\end{enumerate}
\end{theorem}
	\begin{proof}
	If $d_{r_1}(1\otimes a)=0~\&~d_{r_2}(1\otimes b)\neq 0,$ then either $r_2=m-n+1,$ where $m-n$ is odd or $r_2=m+1,$ where $m$ is odd. 
	In this theorem, we consider two cases: (i) $d_{r_3}(1\otimes c)=0$ and (ii)  $d_{r_3}(1\otimes c)\neq 0.$\\
	{\bf Case (i):} $d_{r_3}(1\otimes c)=0.$\\
	First, suppose that   $r_2=m-n+1,$ where $m-n$ is odd. We must have $n<m\leq l,$ and  $d_{m-n+1}(1\otimes b)=c_0t^{\frac{m-n+1}{2}}\otimes a, 0\neq c_0\in \mathbb{Q}.$ So, $d_{m-n+1}(1\otimes bc)=c_0t^{
		\frac{m-n+1}{2}}\otimes ac.$ \\
		If    $d_{m+n-l+1}$ is nontrivial, then  $d_{m+n-l+1}(1\otimes ab)=c_1t^{\frac{n+m-l+1}{2}}\otimes c, 0\neq c_1\in \mathbb{Q},$ and $l$ is even. As $G$ acts freely on $X,$  we must have  $d_{m+n+l+1}(1\otimes abc)=c_2t^{\frac{n+m+l+1}{2}}\otimes 1,0\neq c_2\in \mathbb{Q}.$ Thus $E_{n+m+l+2}^{*,*}=E_{\infty}^{*,*}$ and $E_{\infty}^{p,q} \cong \mathbb{Q}$ for $0 \leq p \leq n+m-l-1$, $p$ even, $q=l;0 \leq p \leq m-n-1, p$ even, $q=n,n+l;0 \leq p \leq n+m+l-1,p$ even, $q=0,$ and zero otherwise . For $n<m\leq l,$ the cohomology groups $H^k(X_G)$ are given by 
	$$H^k(X_G)=
	\begin{cases}
		\mathbb{Q}  & j \leq k(even) < n+j,j=0,m+l,\\
			\mathbb{Q}  &  m+j \leq k(even) < l+j,j=0,n,\\
		\mathbb{Q} \oplus \mathbb{Q} & l \leq k < n+m,k-l~ even,\\
		\mathbb{Q} \oplus \mathbb{Q} & n+j \leq k <m+j, j=0,l,k-(n+j)~ even,\\
		0 & \mbox{otherwise.}
	\end{cases}
	$$
	The permanent cocycles $t\otimes 1, 1 \otimes a,$ $1 \otimes c $ and $1 \otimes ac$ of $E_2^{*,*}$ determine the elements $x \in   E_{\infty}^{2,0}, u \in E_{\infty}^{0,n},$ $v \in E_{\infty}^{0,l}$ 
	and $ ~s \in E_{\infty}^{0,n+l},$ respectively. 
	The total complex Tot$ E_{\infty}^{*,*}$ is given by  $${\mathbb{Q}[x,u,v,s]}/ <x^{\frac{n+m+l+1}{2}}, u^2+\gamma_1v,v^2,s^2,uv+\gamma_2s,us,vs,x^{\frac{m-n+1}{2}}u,x^{\frac{m-n+1}{2}}s, x^{\frac{m+n-l+1}{2}}v>,$$ where deg $x=2,$ deg $u = n$, deg $v = l~\&$  deg $s= n+l$ and  $\gamma_1,\gamma_2 \in \mathbb{Q},$ $\gamma_1=0$ if $l\neq 2n.$  Let $ y \in H^n(X_G), w \in H^{l}(X_G)$ and $ z \in H^{n+l}(X_G)$ such that $i^*(y)=a,i^*(w)=c$ and $i^*(z)=ac,$ respectively. Clearly, $I_1=y^2+b_1x^{n}+b_2x^{\frac{2n-l}{2}}w+b_3x^{\frac{n}{2}}y=0,$ $I_2=w^2+b_4x^{l}=0 ,$ $ I_3=yw+b_5z+b_6x^{\frac{n+l}{2}}=0,$ and $I_4=yz+b_7x^{\frac{n}{2}}z+b_{8}x^{\frac{2n+l}{2}},$ $b_i \in \mathbb{Q}, 1 \leq i \leq 9; $   $b_2=0$ if $l$ odd or  $l>2n, b_3=0$ if $m-1<2n$ or $n$ odd and  $b_6=0$ if $n$ odd, $b_7=0$ if $2n+1>m.$   Then the  cohomology ring of the orbit space $X/G$ is given by $$ { \mathbb{Q}[x,y,w,z]}/{<x^{\frac{n+m+l+1}{2}}, I_j, z^2,wz,x^{\frac{m-n+1}{2}}y,x^{\frac{m-n+1}{2}}z,x^{\frac{m+n-l+1}{2}}w>_{1\leq j \leq 4.}}$$  where deg $x=2,$ deg $y=n,$ deg $w=l$ and deg $z=n+l.$ This realizes possibility (1) when $j'=l.$\\
	If $d_{m+n-l+1}$ is trivial, then $d_{n+m+1}$ must be nontrivial. Therefore, we have $d_{n+m+1}(1\otimes ab)=t^{\frac{n+m+1}{2}}\otimes 1.$ Consequently, $d_{n+m+1}(1\otimes abc)=t^{\frac{n+m+1}{2}}\otimes c.$ Thus $E_{n+m+2}^{*,*}=E_{\infty}^{*,*}, $ and  hence $E_{\infty}^{p,q} \cong \mathbb{Q}$ for $0 \leq p(even) \leq n+m-1$, $q=0,l;0 \leq p(even) \leq m-n-1$, $q=n,n+l,$ and zero otherwise. For $n<m\leq l\leq m+n,$ the cohomology groups $H^k(X_G)$ are given by 
	\\
	$$H^k(X_G)=
	\begin{cases}
		\mathbb{Q}  & 0 \leq k(even) < n;  m \leq k(even) < l,n+m< k < n+l,k-l\mbox{~even},\\
		\mathbb{Q}  & m+l<k\leq n+m+l-1, k-l\mbox{~even},\\
		\mathbb{Q} \oplus \mathbb{Q} &  n\leq k < m, k-n \mbox{~even},l \leq k \leq n+m-1,k-l\mbox{~even},\\
			\mathbb{Q} \oplus \mathbb{Q} & n+l\leq k \leq m+l-1, k-(n+l)\mbox{~even}\\
		0 & \mbox{otherwise,}
	\end{cases}
	$$
	and, for $l>m+n,$ we have 
	$$
	H^k(X_G)=
	\begin{cases}
		\mathbb{Q}  & 0 \leq k(even) < n; j< k < n+j,j=m,l,k-j\mbox{~even}\\
			\mathbb{Q}  & m+l< k\leq m+n+l-1,k-l\mbox{~even}\\
		\mathbb{Q} \oplus \mathbb{Q} & n+j\leq k < m+j,j=0\mbox{~and~} k,n \mbox{~even};j=l \mbox{~and~}k-{n+l}\mbox{~even}\\

		0 & \mbox{otherwise.}
	\end{cases}
$$
	The permanent cocycles $t\otimes 1, 1 \otimes a,$ $1 \otimes c $ and $1 \otimes ac$ of $E_2^{*,*}$ determine the elements  $x\in   E_{\infty}^{2,0},  u \in E_{\infty}^{0,n},$ $v \in E_{\infty}^{0,l},$ and
	$s \in E_{\infty}^{0,n+l},$ respectively. Thus the total complex Tot$ E_{\infty}^{*,*}$ is given by \begin{center}
		$\mathbb{Q}[x,u,v,s]/<x^{\frac{n+m+1}{2}},u^2+\gamma_1v,v^2,s^2,uv+\gamma_2s,us,vs,x^{\frac{m-n+1}{2}}u,x^{\frac{m-n+1}{2}}s>,$
	\end{center} 
 where deg $x=2,$ deg $u$ = $n$, deg $v$ = $l$ $\&$  deg $s= n+l$ and $\gamma_1,\gamma_2 \in \mathbb{Q}, \gamma_1=0 $ if $l\neq 2n.$ Let $y \in H^n(X_G), w \in H^{l}(X_G)$ and $ z \in H^{n+l}(X_G)$ such that $i^*(y)=a,i^*(w)=c$ and $i^*(z)=ac,$ respectively. Clearly, we have  $I_1=y^2+b_1x^{n}+b_2x^{\frac{2n-l}{2}}w+b_3x^{\frac{n}{2}}y=0,$ $I_2=w^2+b_4x^{\frac{l}{2}}w+b_5x^{\frac{l-n}{2}}z=0,$ $I_3=yw+b_6z+b_7x^{\frac{n}{2}}w=0,$ and $I_4=yz+b_8x^{\frac{n}{2}}z+b_{9}x^{n}w=0,$  $b_i \in \mathbb{Q}, 1 \leq i \leq 9;$ $ b_2=0$ if $l>2n$ or $l$ odd, $ b_3=0$ if $m<2n+1$ or $n$ odd,    $b_4=0$ if $l$ odd or $l>n+m-1,$ $b_7=0$ if $n$ odd and $b_8=0$ if $m<2n+1.$  Thus the  cohomology ring of the orbit space $X/G$ is given by $${ \mathbb{Q}[x,y,w,z]}/{<x^{\frac{n+m+1}{2}}, I_j, z^2,wz,x^{\frac{m-n+1}{2}}y,x^{\frac{m-n+1}{2}}z>_{1\leq j \leq 4}},$$ where deg $x=2,$ deg $y=n,$ deg $w=l$ and deg $z=n+l.$ This realizes possibility (1) when $j'=0.$\\ 
	Next, we consider $r_2=m+1,$ where $m$ is odd. We have $d_{m+1}(1\otimes b)=c_0t^{\frac{m+1}{2}}\otimes 1.$ So, we get  $d_{m+1}(1\otimes ab)=c_0t^{\frac{m+1}{2}}\otimes a $, $d_{m+1}(1\otimes bc)=c_0t^{\frac{m+1}{2}}\otimes c$ and $d_{m+1}(1\otimes abc)=c_0t^{\frac{m+1}{2}}\otimes ac.$ Clearly, $d_r=0$ for all $r>m+1.$ Thus we get $E_{m+2}^{*,*}=E_{\infty}^{*,*}.$ \\ If $n\leq m<l$ or $n<m =l,$ then  $E_{\infty}^{p,q} \cong \mathbb{Q}$ for $0 \leq p \leq m, p$ even $\&$ $q=0,n,l$ or $n+l,$ and zero otherwise.\\ If $n= m=l$ then  $E_{\infty}^{p,q} \cong \mathbb{Q}$ for $0 \leq p \leq n-1, p$ even $\&$ $q=0$ or $l,$  and $E_{\infty}^{p,n} \cong \mathbb{Q}\oplus \mathbb{Q}$ for $0 \leq p \leq n-1, p$ even, and   zero otherwise. For $n\leq m\leq l$  and  $l< m+n,$ the resulting cohomology groups of 
	$X/G$ are as follows. \\
	$$ H^k(X_G)=
	\begin{cases}
		\mathbb{Q}  & 0 \leq k(even) < n; m \leq k < l,k-n ~\mbox{even},\\
		\mathbb{Q}  & n+m<k<n+l, k-l ~\mbox{ even}\\
		\mathbb{Q}  & m+l-1< k< m+n+l-1,k-(n+l)~\mbox{ even}\\
		\mathbb{Q} \oplus \mathbb{Q} & n\leq k < m,k-n~\mbox{even}; l \leq k \leq n+m-1,k-l~\&~ k-n \mbox{ even}\\
		\mathbb{Q} \oplus \mathbb{Q} & n+l \leq k \leq m+l-1, k-(n+l) \mbox{ even}\\
		0 & \mbox{otherwise.}
	\end{cases}
	$$ and, for $n\leq m<l$ and $l > m+n, $ we have 
	$$
	H^k(X_G)=
	\begin{cases}
		\mathbb{Q}  & 0 \leq k(even) < n; m \leq k \leq n+m-1 ,k-n ~\mbox{even},\\
		\mathbb{Q}  & l\leq k<n+l, k-l ~\mbox{ even},\\
	\mathbb{Q}  & m+l<k<n+m+l, k-(n+l) ~\mbox{ even},\\
	\mathbb{Q} \oplus \mathbb{Q} & n\leq k < m,k-n~\mbox{even},\\
	\mathbb{Q} \oplus \mathbb{Q} &  n+l \leq k \leq m+l-1,k-l~\&~ k-(n+l) \mbox{ even},\\
	\mathbb{Q} \oplus \mathbb{Q} & n+l \leq k \leq m+l-1, k-(n+l) \mbox{ even},\\
	0 & \mbox{otherwise.}
	\end{cases}
	$$
	The permanent cocycles $t\otimes 1, 1 \otimes a,$ $1 \otimes $c and $1 \otimes ac$ of $E_2^{*,*}$ determine the elements $ x\in   E_{\infty}^{2,0}, u \in E_{\infty}^{0,n},$ $v \in E_{\infty}^{0,l}$ and
	$s \in E_{\infty}^{0,n+l},$ respectively.  Thus the total complex is given by \begin{center}
		Tot $E_{\infty}^{*,*} = \mathbb{Q}[x,u,v,s]/<x^{\frac{m+1}{2}},u^2+\gamma_1v+\gamma_2v,v^2,s^2,uv+\gamma_3s,us,vs>,$	\end{center} where deg $x=2,$ deg $u= n,$ deg $v= l,$   deg $s= n+l,$ and ${\gamma_1,\gamma_2~\&~\gamma_3 \in \mathbb{Q}},$ $\gamma_1=0$ if $l \neq 2n~\&~\gamma_2=0$ if $n\neq l.$ 
	Let $y \in H^n(X_G), w \in H^l(X_G)$ and $ z \in H^{n+l}(X_G)$ such that $i^*(y)=a,i^*(w)=c,$ and $i^*(z)=ac,$ respectively. We have $I_1= y^2+b_1x^{\frac{2n-l}{2}}w+b_2x^{\frac{n}{2}}y+b_3x^{n}+b_4z=0,$ $I_2= w^2+b_5x^{\frac{l-n}{2}}z=0,$   $I_3= yw+b_6z+b_7x^{\frac{n}{2}}w=0,$ and $I_4=yz+b_8x^{\frac{n}{2}}z+b_9x^{n}w=0,$ $b_i \in \mathbb{Q}, 0\leq i \leq 9;$ $b_1=0 $ if $2n>m+l-1,$ $b_2=0$ if $n$ odd,$b_3=0$ if $2n+1>m,$ $b_4=0$ if $n\neq l,$  $b_5=0$ if $l>m+n-1$ and $b_7=b_8=0$ if $n$ odd and $b_9=0$ if  $2n+1>m.$   Thus the  cohomology ring of the orbit space $X/G$ is given by $$ \mathbb{Q}[x,y,w,z]/<x^{\frac{m+1}{2}}, I_j, z^2,wz>_{1\leq j \leq 4} ,$$ where  deg $x=2,$ deg $y=n,$ deg $w=l,$ and  deg $z=n+l.$	   This realizes possibility (2). \\ 
	{\bf Case (ii):}  $d_{r_3}(1\otimes c)\neq 0.$\\ 
	First, suppose that $r_2=m-n+1,$ where $m-n$ is odd. We must have  $n <m \leq l,$ and  $d_{m-n+1}(1\otimes b)=c_0t^{\frac{m-n+1}{2}}\otimes a, 0\neq c_0\in \mathbb{Q}.$ Consequently, $d_{m-n+1}(1\otimes bc)=t^{\frac{m-n+1}{2}}\otimes ac.$ In this case, we have $r_3=l-m-n+1$ or $l+1.$ \\If  $r_3=l+1,$ where $l$ is odd, then $d_{l+1}(1\otimes c) = c_0t^{\frac{l+1}{2}}\otimes 1$ and $d_{l+1}(1\otimes abc) = c_0t^{\frac{l+1}{2}}\otimes ab.$ Thus $E_{l+2}^{*,*}=E_{\infty}^{*,*}.$ If $n<m<l,$  then $E_{\infty}^{p,q}\cong \mathbb{Q},$ for $0\leq p \leq m-n-1,p$ even $q=n$ or $n+l; 0 \leq p \leq l-1,~p$ even $q=0$ or $n+m,$ and  zero otherwise. \\ For $l< m+n,$ the cohomology groups of $X/G$ are as follows.
	$$H^k(X_G)=
	\begin{cases}
		\mathbb{Q}  & 0 \leq k(even) < n; m-1< k(even)\leq l-1,\\
		\mathbb{Q}  & m+n \leq k < n+l,k-(n+m)\mbox{ even},\\
		\mathbb{Q}  & m+l<k< m+n+l,k-(n+m)\mbox{ even},\\
		\mathbb{Q} \oplus \mathbb{Q} & n\leq k(even) \leq m-1,\\
			\mathbb{Q} \oplus \mathbb{Q} & n+l\leq k \leq m+l-1, k-(n+l)~\&~k-(n+m) \mbox{ even},\\
		0 & \mbox{otherwise},
	\end{cases}
	$$  and, for $l \geq m+n, $ we have 
	$$ H^k(X_G)=
	\begin{cases}
		\mathbb{Q}  & 0 \leq k(even) < n; m < k < n+m,k~\mbox{even} ,\\
		\mathbb{Q}  & 	j < k < n+j,j=l~\mbox{or}~m+l-1~\&~k-(n+m)~\mbox{even},\\
		\mathbb{Q} \oplus \mathbb{Q} & n \leq k \leq m,n~\mbox{even};m+n \leq k \leq l-1,k~\&~k-(n+m)\mbox{~even} \\
			\mathbb{Q} \oplus \mathbb{Q} & n+l \leq k \leq m+l-1,k-(n+l)\mbox{~even} \\
		0 & \mbox{otherwise.}
	\end{cases}
	$$
	If $n<m=l,$ then $E_{\infty}^{p,q}\cong \mathbb{Q},$ for $0\leq p(even) \leq m-1,q=0;0\leq p \leq m-n,q=n;m-n<p(even)\leq m , q=n, $ $m-n-1\leq p(even)\leq m-1,q=n+m$ and $E_{\infty}^{p,q}\cong \mathbb{Q}\oplus \mathbb{Q},$ for $ 0 \leq p(even) \leq m-n-1,q=n+m,$ and zero otherwise.\\
	Thus the resulting cohomology groups of $X_G$ are given by 
	$$
	H^k(X_G)=
	\begin{cases}
		\mathbb{Q}  & 0 \leq k(even) < n; 2m-1<k\leq 2m+n-1,k-(n+m)~\mbox{even}\\
		\mathbb{Q} \oplus \mathbb{Q} & n\leq k \leq m-1,n~\mbox{even}, n+m\leq k \leq 2m-1,k-(n+m)~\mbox{even}\\
		0 & \mbox{otherwise},
	\end{cases}
	$$ 
	The permanent cocycles $t\otimes 1, 1 \otimes a,$ $1 \otimes ab$ and $1 \otimes ac$ of $E_2^{*,*}$ determine the elements $x\in E_{\infty}^{2,0}, u \in E_{\infty}^{0,n},$ $v \in E_{\infty}^{0,n+m}$ and
	$s \in E_{\infty}^{0,n+l},$ respectively.  Thus for $n<m\leq l,$ the total complex is given by \begin{center}
		Tot $E_{\infty}^{*,*} = \mathbb{Q}[x,u,v,s]/<x^{\frac{l+1}{2}},u^2,v^2+\gamma_1s,s^2,uv+\gamma_2s,us,vs,x^{\frac{m-n+1}{2}}u,x^{\frac{m-n+1}{2}}s>,$
	\end{center} where deg $x=2,$ deg $u= n,$ deg $v= n+m,$  deg $s= n+l$ and ${\gamma_1, \gamma_2  \in \mathbb{Q}};$ $\gamma_1=0$ if $l \neq n+2m$ and  $\gamma_2=0$ if $l \neq n+m.$ Let $y \in H^n(X_G), w \in H^{n+m}(X_G)$ and $ z \in H^{n+l}(X_G)$ such that $i^*(y)=a,i^*(w)=c,$ and $i^*(z)=ac,$ respectively. Clearly,  we have $I_1=y^2+b_1x^{n}+b_2x^{\frac{n}{2}}y=0,$ $I_2=w^2+b_3x^{\frac{n+m}{2}}w+b_4x^{n+m}+b_5x^{\frac{2m+n-l}{2}}z=0,$ $ I_3=yw+b_6x^{\frac{n}{2}}w+b_7x^{\frac{2n+m}{2}}+b_8x^{\frac{n+m-l}{2}}z=0,$ $I_4=yz+b_9x^{\frac{n+l-m}{2}}w+b_{10}x^{\frac{n}{2}}z=0,b_i\in \mathbb{Q}, 1\leq i \leq 10;$ $ b_1=0 ~\mbox{if}~ l-1<2n,b_2=0$ if $n$ odd or $ m-n<n+1;$ $ b_3=0 ~\mbox{if}~ l-1<n+m ~\mbox{or}~ m=l, b_4=0 ~\mbox{if}~ l-1<2n+2m ~\mbox{or}~ m=l,$  $ b_5=0$ if $l-1<m+2n$ or $m=l,$$b_6=0$ if $n$ odd, $b_7=0 ~\mbox{if}~ l-1<2n+m,b_8=0 $ if $l<2n+1$ and  $ b_{10}=0 ~\mbox{if}~ 2n+1>m.
	$ Thus the  cohomology ring of the orbit space $X/G$ is given by  $$\mathbb{Q}[x,y,w,z]/<x^{\frac{l+1}{2}}, I_j, z^2,wz,x^{\frac{m-n+1}{2}}y,x^{\frac{m-n+1}{2}}z>_{1\leq j \leq 4},$$ where deg $x=2,$  deg $y=n,$  deg $w=n+m$ and  deg $z=n+l.$
	This realizes possibility (3) when $j'=0.$\\
	If $r_3=l-m-n+1,$ then we must have  $l>n+m$ and $d_{l-n-m+1}(1\otimes c)=t^{\frac{l-n-m+1}{2}} \otimes ab.$ As $G$ acts freely on $X,$ we get  $d_{n+m+l+1}(1\otimes abc)= t^{\frac{n+m+l+1}{2}}\otimes 1.$ Thus $E_{n+m+l+2}^{*,*}=E_{\infty}^{*,*}.$ We get $E_{\infty}^{p,q}\cong\mathbb{Q},$ for  $0 \leq p(even) \leq n+m+l-1, q=0;0\leq p(even) \leq m-n-1,q=n,n+l; 0\leq p(even) \leq l-m-n-1,q=n+m,$ and zero otherwise. Consequently, the cohomology groups of $X_G$ are given by
	$$
	H^k(X_G)=
	\begin{cases}
		\mathbb{Q}  & 0 \leq k(even) < n; l+j \leq k(even) \leq  n+j-1,j=0,m,\\
		\mathbb{Q} \oplus \mathbb{Q} & n+j \leq k(even) \leq m+j-1,j=0,l~\&~k-j~\mbox{even},\\
		\mathbb{Q} \oplus \mathbb{Q} & n+m \leq k(even) \leq l-1,k-(n+m)~\mbox{even}, \\
		0 & \mbox{otherwise.}
	\end{cases}
	$$
	The permanent cocycles $t\otimes 1, 1 \otimes a,$ $1 \otimes ab$ and $1 \otimes ac$ of $E_2^{*,*}$  determine the elements $x\in E_{\infty}^{2,0}, u \in E_{\infty}^{0,n},$ $v \in E_{\infty}^{0,n+m}$ and 
	$s \in E_{\infty}^{0,n+l},$ respectively.  Thus  the total complex Tot $E_{\infty}^{*,*}$ is given by \begin{center}
		$\mathbb{Q}[x,u,v,s]/<x^{\frac{n+m+l+1}{2}},u^2,v^2+\gamma_1s,s^2,uv,us,vs,x^{\frac{m-n+1}{2}}u,x^{\frac{m-n+1}{2}}s,x^{\frac{l-m-n+1}{2}}v>,$
	\end{center}   where deg $x=2,$ deg $u= n,$ deg $v= n+m,$  deg $s= n+l$ and ${\gamma_1  \in \mathbb{Q}},$ $\gamma_1=0$ if $l \neq n+2m.$ Let $y \in H^n(X_G), w \in H^{n+m}(X_G)$ and $ z \in H^{n+l}(X_G)$ such that $i^*(y)=a,i^*(w)=c,$ and $i^*(z)=ac,$ respectively. Clearly, $I_1=y^2+b_1x^{n} +b_2x^{\frac{n}{2}}y=0,$ $I_2=w^2+b_3x^{n+m}+b_4x^{\frac{n+m}{2}}w=0,$ $I_3=yw+b_5x^{\frac{2n+m}{2}}+b_6x^{\frac{n}{2}}w=0,$ $I_4=yz+b_7x^{\frac{2n+l}{2}}+b_8x^{\frac{n}{2}}z,b_i\in \mathbb{Q}, 1\leq i \leq  8;$ $b_2=0$ if $n$ odd or $m-1<2n; $ $b_4=0$ if $l<n+m+1, $ $b_6=0$ if  $l<2n+m+1,$ and $b_8=0$ if $n$ odd or  $m<2n+1.$ So, the cohomology ring $H^*(X_G)$ is given by $${\mathbb{Q}[x,y,w,z]}/{<x^{\frac{n+m+l+1}{2}}, I_j, z^2,wz,x^{\frac{m-n+1}{2}}y,x^{\frac{m-n+1}{2}}z,x^{\frac{l-m-n+1}{2}}w>_{1\leq j \leq 4}}$$ where deg $x=2,$  deg $y=n,$  deg $w=n+m$ and deg $x=n+l.$ This realizes possibility (3) when $j'=m+n.$

	Finally, suppose that  $r_2=m+1.$ Clearly, either  $r_3=l-n+1$ or $l+1.$\\ First, we consider $r_3=l-n+1.$ If $m<l-n,$ then $d_{r_3}(1\otimes c)=0,$ which contradicts our hypothesis. So, we get $l-n\leq m.$ As $d_{l-n+1}$ is nontrivial, we get $d_{l-n+1}(1\otimes c)=c_2t^{\frac{l-n+1}{2}}\otimes a,0\neq c_2\in \mathbb{Q} $ and $d_{l-n+1}(1\otimes bc)=c_2t^{\frac{l-n+1}{2}}\otimes ab.$ Also,  we have $d_{m+1}(1\otimes b)=c_1t^{\frac{m+1}{2}}\otimes 1~\&~d_{m+1}(1\otimes abc)=c_1t^{\frac{m+1}{2}}\otimes  ac.$ Thus $E_{m+2}^{*,*}=E_{\infty}^{*,*}.$ If $n\leq m<l$ and $l-n<m,$  then  $E_{\infty}^{p,q}\cong \mathbb{Q},$ for $0\leq p(even) \leq m-1,q=0$ or $n+l; 0\leq p(even) \leq l-n-1, q=n$ or $n+m,$ and zero otherwise. Thus the cohomology groups   are given by 
	$$H^k(X_G)=
	\begin{cases}
		\mathbb{Q}  & 0 \leq k(even) < n; m< k\leq l-1,k-n\mbox{ even},\\
\mathbb{Q}  &  m+n \leq k < n+l,k-(n+m)~\mbox{even,}\\
		\mathbb{Q}  & 	m+l\leq k< m+n+l,k-(n+l)~\mbox{even},\\
		\mathbb{Q} \oplus \mathbb{Q} & n+j\leq k \leq m+j-1, j=0~\&~n~\mbox{even},j=l~\&~k-(n+m)~\mbox{even},\\
		0 & \mbox{otherwise.}
	\end{cases}
	$$ The permanent cocycles $t\otimes 1, 1 \otimes a,$ $1 \otimes ab$ and $1 \otimes ac$ of $E_2^{*,*}$  determine the elements $x\in   E_{\infty}^{2,0},$ $u \in E_{\infty}^{0,n},$ $v \in E_{\infty}^{0,n+m}$ and $s \in E_{\infty}^{0,n+l},$ respectively.  Then the total complex Tot $E_{\infty}^{*,*}$ is given by \begin{center}
		$\mathbb{Q}[x,u,v,s]/<x^{\frac{m+1}{2}},u^2+\gamma_1v,v^2,s^2,uv,us,vs,x^{\frac{l-n+1}{2}}u,x^{\frac{l-n+1}{2}}v>,$ 
	\end{center}  where deg $x=2,$ deg $u= n,$ deg $v= n+m,$  deg $s= n+l$ and  ${\gamma_1  \in \mathbb{Q}},$ $\gamma_1=0$ if $ n<m.$  Let  $y \in H^n(X_G), w \in H^{n+m}(X_G)$ and $ z \in H^{n+l}(X_G)$ such that $i^*(y)=a,i^*(w)=c,$ and $i^*(z)=ac,$ respectively. We have $I_1=y^2+b_1x^{n}+b_2x^{\frac{n}{2}}y +b_3w=0,I_2=yw +b_4x^{\frac{n}{2}}w+b_5x^{\frac{m+n-l}{2}}z=0,$ and $ I_3= yz+b_6x^{\frac{n}{2}}z+b_7x^{\frac{l+n-m}{2}}w=0,$ $b_i\in \mathbb{Q}, 1 \leq i \leq 7;$  $b_2=b_4=0$ if $l<2n+1,$ $b_3=0$ if $n < m$ and $ b_1=b_7=0$ if  $m<2n+1.$ Therefore, the cohomology ring $H^*(X_G)$ is given by	$$\mathbb{Q}[x,y,w,z]/<x^{\frac{m+1}{2}}, I_j,w^2,
	z^2,wz,x^{\frac{l-n+1}{2}}y,x^{\frac{l-n+1}{2}}w>_{1\leq j \leq 3}$$
	deg $x=2,$ deg $y=n,$ deg $w=n+m,$ and  deg $z=n+l.$  This realizes possibility (4). \\
	Next, if $n\leq m<l,$ then the cohomology groups and  cohomology algebra are same in the case (i) when $r_2=m+1$ and $l=m+n.$\\ If $n<m=l$ and $d_{l-n+1}(1\otimes c)\neq 0,$ then the cohomology groups and cohomology algebra are same as in the   case (ii) when $n< m=l.$\\ Next, if $n<m=l$ and $d_{l-n+1}(1\otimes c)= 0,$ then we must have $r_2=r_3=m+1$ and hence, $d_{m+1}(1\otimes b)=d_{m+1}(1 \otimes c)=t^{\frac{m+1}{2}}\otimes 1.$  The cohomology groups and cohomology algebra are same as in the case (i) when $r_2=m+1$ and $n< m=l.$\\ Finally, if $r_3=l+1,$ then we must have $n=m=l,$ and  $r_2=r_3=n+1$ and hence, $d_{n+1}(1\otimes b)=d_{n+1}(1 \otimes c)=t^{\frac{n+1}{2}}\otimes 1.$ The cohomology groups and cohomology algebra are same as in  the case (i) when $n=m=l.$
		\end{proof}

\begin{theorem}\label{thm 3.8}
	Let $G=\mathbb{S}^1$ act freely on a finitistic space $X \sim_\mathbb{Q} \mathbb{S}^n \times \mathbb{S}^m \times \mathbb{S}^l, $ where $n\leq m \leq l.$ If $d_{r_1}(1\otimes a)=d_{r_2}(1\otimes b)=0 $ and $d_{r_3}(1\otimes c) \neq 0,$ then  $H^*(X/G)$ is isomorphic to one of the following graded commutative algebras:
	\begin{enumerate}
		\item \label{1} $\mathbb{Q}[x,y,w,z]/I,$ where $I$ is homogeneous ideal given by: \\ $I=<Q(x), I_j,x^{\frac{l-m-n+1}{2}}z,c_0x^{\frac{l+n-m+1}{2}}w,c_1x^{\frac{l+m-n+1}{2}}y,c_2x^{\frac{l+1}{2}}y,c_3x^{\frac{l+1}{2}}w>_{1\leq j \leq 6},$\\
		where deg $x=2,$ deg $y=n$, deg $w=m~\&$ deg $z=n+m$ and $ I_1 =y^2+a_1x^{n}+a_2x^{\frac{n}{2}}y+a_3x^{\frac{2n-m}{2}}w+a_4z,
		I_2 =w^2+a_5x^{m}+a_6x^{\frac{2m-n}{2}}y+a_7x^{\frac{m}{2}}w+a_8x^{\frac{m-n}{2}}z,
		I_3= z^2+a_9x^{n+m} +a_{10}x^{\frac{n+2m}{2}}y+a_{11}x^{\frac{2n+m}{2}}w+a_{12}x^{\frac{n+m}{2}}z,
		I_4 =yw+a_{13}x^{\frac{n+m}{2}}+a_{14}x^{\frac{m}{2}}y+a_{15}x^{\frac{n}{2}}w+a_{16}z,
		I_5 =yz+a_{17}x^{\frac{2n+m}{2}}+a_{18}x^{\frac{n+m}{2}}y+a_{19}x^{n}w+a_{20}x^{\frac{n}{2}}z~\&~$ $
		I_6 =wz+a_{21}x^{\frac{2m+n}{2}}+a_{22}x^{m}y+a_{23}x^{\frac{n+m}{2}}w+a_{24}x^{\frac{m}{2}}z, a_i\in \mathbb{Q}, 1\leq i \leq 24;$  $a_4=0$ if $n<m, a_8=0$ if $l<2m+1,$ $a_{12}=0$ if $l<2m+2n+1,$   $a_{20}=0$ if $l<2n+m+1,$ and  $a_{24}=0$ if $l<n+2m+1;$   $Q(x)=x^{\frac{l+1+j'}{2}}, j'=n+m,m$ or $n.$\\
		If $j'=n+m,$ then either \{$c_0=c_1=1~\&~ c_2=c_3=0$ with $a_7=0$ if $n+l<2m+1,a_{10}=0$ if $l<2n+m+1,a_{11}=0$ if $l<n+2m+1~\&~a_{23}=0$ if $l<2m+1$\} or \{$c_0=c_1=0 ~\&~ c_2=c_3=1$ with $a_{10}=0$ if $l<2n+m+1,a_{11}=0$ if  $l<n+2m+1~\&~a_{22}=0$ if $l<2m+1$\}.\\
		If $j'=m,$ then $c_0=1~\&~ c_1=c_2=c_3=0$ with $a_7=0$ if $n+l<2m+1, a_9=0$ if $l<2n+m+1,a_{11}=0$ if $l<n+2m+1~\&~$$a_{23}=0$ if $l<2m+1.$\\
		If $j'=n,$ then $c_1=1~ \&~ c_0=c_2=c_3=0$ with $a_5=0$ if $n+l<2m+1,a_{9}=0$ if $l<2m+n+1,a_{10}=0$ if $l<2n+m+1~\&~$$a_{21}=0$ if $l<2m+1.$  
		\item $\mathbb{Q}[x,y,w,z]/<Q(x), I_j,c_0x^{\frac{m+l-n+1}{2}}y,x^{\frac{l-m+1}{2}}z,x^{\frac{l-m+1}{2}}w>_{1\leq j \leq 6},$\\ where deg $x=2,$ deg $y=n$, deg $w=m~\&$ deg $z=n+m$ and ${I_j}'s,  1\leq j \leq 6$ are same as in the possibility $(\ref{1}),$ with $a_3=0$ if $l<2n+1, a_4=0$ if $n < m,a_7=a_{24}=0 $ if $l<2m+1,a_{8}=0$ if $n+l<2m+1,a_{11}=0$ if $l<2n+2m+1,a_{12}=a_{23}=0$ if $l<n+2m+1,a_{15}=a_{20}=0$ if $l<m+n+1~\&~a_{19}=0$ if $l<2n+m+1,$ and  $Q(x)=x^{\frac{m+l+j'+1}{2}},j'=0 $ or $n.$\\
		If $j'=0,$ then $c_0=0 $ with $a_9=0$ if $l<2n+m,a_{10}=a_{21}=0$ if $l<n+m~\&~ a_{17}=0$ if  $l<2n.$\\ If $j'=n,$ then $c_0=1 $ with $ a_9=a_{22}=0$ if $l<n+m+1,a_{10}=0$ if $l<2n+m+1$$~\&~ a_{18}=0$ if  $l<2n+1.$ 
		\item   $\mathbb{Q}[x,y,w,z]/<Q(x), I_j,c_0x^{\frac{n+l-m+1}{2}}w,x^{\frac{l-n+1}{2}}y,x^{\frac{l-n+1}{2}}z>_{1\leq j \leq 6},$\\ where deg $x=2,$ deg $y=n$, deg $w=m~\&$ deg $z=n+m$ and  ${I_j}'s, 1\leq j \leq 6$ are same as in the possibility $(\ref{1}),$ with $a_2=a_{20}=0$ if $l<2n+1, a_4=0$ if $n < m,a_6=0 $ if $l<2m+1,a_{10}=0$ if $l<2n+2m+1,a_{14}=a_{24}=0$ if $l<n+m-1~\&$ $a_{18}=0$ if $l<2n+m+1$ and $Q(x)=x^{n+1+j'},j'=0 $ or $m.$ \\
		If $j'=0,$ then $c_0=0 $ with $ a_5=0$ if $n+l<2m+1,a_{9}=a_{12}=0$ if $l<2m+n+1,a_{11}=a_{17}=0$ if  $l<n+m+1,a_{21}=0$ if $l<2m+1~\&~ a_{22}=0$ if $l<2n+m+1.$\\  If $j'=m,$ then $c_0=1 $ with $a_{7}=0$ if $n+l<2m+1, a_9=a_{19}=0$ if $l<n+m-1,a_{11}=a_{22}=0$ if $l<2m+n+1,a_{12}=0$ if $l<2n+m+1~\&~ a_{23}=0$ if  $l<2m+1.$ 
		\item $\mathbb{Q}[x,y,w,z]/<x^{\frac{l+1}{2}}, I_j>_{1\leq j \leq 6},$ where $l$ odd, deg $x=2,$ deg $y=n$, deg $w=m~\&$ deg $z=n+m$ and ${I_j}'s, 1\leq j \leq 6$ are same as in the possibility $(\ref{1}),$ with $a_1=a_{19}=0$ if $l<2n+1,$ $a_4=0 $ if $ n<m, a_5=a_{22}=0$ if $l<2m+1,a_{6}=0$ if $n+l<2m+1, a_{9}=0$ if $l<2n+2m+1,$ $a_{10}=a_{21}=0$ if $l<2m+n+1,$ $a_{11}=a_{17}=0$ if $l<2n+m+1~\&~$$a_{12}=a_{13}=a_{18}=a_{23}=0$ if $l<n+m+1.$   
	\end{enumerate}
	
\end{theorem}
\begin{proof}
		If $d_{r_1}(1\otimes a)=d_{r_2}(1\otimes b) = 0$ and  $d_{r_3}(1\otimes c) \neq 0 ,$ then we have following four cases: (i) $r_3=l-m-n+1,$ where $l-m-n$ odd (ii) $r_3=l-m+1,$ where $l-m$ odd (iii) $r_3=l-n+1,$ where $l-n$ odd and (iv) $r_3=l+1,$ where $l$ odd.\\
	{\bf Case (i):} $r_3=l-m-n+1,~l-m-n$ odd. \\ Clearly, $n \leq m <l$ and $n+m<l.$  We have  $d_{l-m-n+1}(1\otimes c)= c_0t^{\frac{l-m-n+1}{2}} \otimes ab,$ where $0 \neq c_0\in \mathbb{Q}.$\\ First, we assume that $d_{n+l-m+1}$ is nontrivial. So, we have $d_{n+l-m+1}(1\otimes ac)= c_1t^{\frac{n+l-m+1}{2}}\otimes b$ for some nonzero rational $c_1.$ Now, we have either $d_{m+l-n+1}(1\otimes bc)=0$ or $d_{m+l-n+1}(1\otimes bc)=c_2t^{\frac{m+l-n+1}{2}}\otimes a ,0\neq c_2\in \mathbb{Q}.$ \\Let $d_{m+l-n+1}(1\otimes ac)=c_2t^{\frac{m+l-n+1}{2}}\otimes a .$ So, we must have  $d_{n+m+l+1}(1\otimes abc)=c_3t^{\frac{n+m+l+1}{2}}\otimes 1, 0\neq c_3\in \mathbb{Q}.$ Thus $E_{n+m+l+2}^{*,*}=E_{\infty}^{*,*}.$ For $n<m< l,$ we have $E_{\infty}^{p,q}\cong \mathbb{Q},$ $ 0 \leq p(even) \leq n+m+l-1,q=0; 0 \leq p(even) \leq m+l-n-1,q=n; 0 \leq p(even) \leq n+l-m-1,q=m 
	$ and $0 \leq p(even) \leq l-m-n-1,q=n+m,$ and zero otherwise. For $n=m<l,$ we have $E_{\infty}^{p,q} \cong \mathbb{Q},$  $ 0 \leq p(even) \leq 2n+l-1,q=0$ and $ 0 \leq p(even) \leq l-2n-1,q=2n; E_{\infty}^{p,q} \cong \mathbb{Q} \oplus \mathbb{Q},$  $ 0 \leq p(even) \leq l-1,q=n,$ and  zero otherwise. \\ For $n<m<l,$ the cohomology groups are of $ X_G$ given by 
	\[
	H^k(X_G) =
	\begin{cases}
		\mathbb{Q}  & 0 \leq k(even) < n; m+l < k(even) < m+n+l,\\
		\mathbb{Q} \oplus \mathbb{Q} & n \leq k < m;n+l < k < m+l,n \mbox{~even},\\
		\mathbb{Q}\oplus \mathbb{Q}\oplus \mathbb{Q} & m \leq k < m+n; l< k \leq n+l-1,m~\mbox{even},\\
		\mathbb{Q}\oplus \mathbb{Q}\oplus \mathbb{Q}\oplus \mathbb{Q} & m+n \leq k \leq l-1,n~\&~m~\mbox{even}.\\
		0 & \mbox{otherwise,}
	\end{cases}
	\]
	and, for $n=m<l,$ we have
	\[
	H^k(X_G) =
	\begin{cases}
		\mathbb{Q}  & 0 \leq k < n; n+l < k \leq 2n+l, k~\mbox{even},\\
		\mathbb{Q}\oplus \mathbb{Q}\oplus \mathbb{Q} & n \leq k < 2n; l< k \leq n+l,k~\&~ n~\mbox{even},\\
		\mathbb{Q}\oplus \mathbb{Q}\oplus \mathbb{Q}\oplus \mathbb{Q} & 2n \leq k \leq l,k~\& ~n~\mbox{even},\\
		0 & \mbox{otherwise.}
	\end{cases}
	\]	The permanent cocycles $t\otimes 1, 1 \otimes a,$ $1 \otimes b$ and $1 \otimes ab$ of $E_2^{*,*}$ determine the elements   $x\in   E_{\infty}^{2,0},$ $u \in E_{\infty}^{0,n},$ $v \in E_{\infty}^{0,m}$  and $s \in E_{\infty}^{0,n+m},$ respectively. Thus the total complex \\Tot $E_{\infty}^{*,*}=\mathbb{Q}[x,u,v,s]/I,$ where ideal $I$ is given by \\  $<x^{\frac{n+m+l+1}{2}},u^2+\gamma_1v+\gamma_2s,v^2+\gamma_3s,s^2,uv+\gamma_4s,us,vs,x^{\frac{l+m-n+1}{2}}u,x^{\frac{l+n-m+1}{2}}v,x^{\frac{l-m-n+1}{2}}s>$ where  deg $x=2,$ deg $u= n,$ deg $v= m,$  deg $s= n+m,$ and ${\gamma_i \in \mathbb{Q}}, 1\leq i \leq 4;$ $\gamma_1=0$ if $m \neq 2n,$ and $\gamma_2=\gamma_3=0$ if $n < m.$ Let $y \in H^n(X_G), w \in H^{m}(X_G)$ and $ z \in H^{n+m}(X_G)$ such that $i^*(y)=a,i^*(w)=b,$ and $i^*(z)=ab,$ respectively. Clearly,
	$I_1 =y^2+b_1x^{n}+b_2x^{\frac{n}{2}}y+b_3x^{\frac{2n-m}{2}}w+b_4z =0,
	I_2 =w^2+b_5x^{m}+b_6x^{\frac{2m-n}{2}}y+b_7x^{\frac{m}{2}}w+b_8x^{\frac{m-n}{2}}z =0,
	I_3 = z^2+b_9x^{n+m} +b_{10}x^{\frac{n+2m}{2}}y+b_{11}x^{\frac{2n+m}{2}}w+b_{12}x^{\frac{n+m}{2}}z =0,
	I_4 =yw+b_{13}x^{\frac{n+m}{2}}+b_{14}x^{\frac{m}{2}}y+b_{15}x^{\frac{n}{2}}w+b_{16}z =0,
	I_5 =yz+b_{17}x^{\frac{2n+m}{2}}+b_{18}x^{\frac{n+m}{2}}y+b_{19}x^{n}w+b_{20}x^{\frac{n}{2}}z =0~\&~
	I_6 =wz+b_{21}x^{\frac{2m+n}{2}}+b_{22}x^{m}y+b_{23}x^{\frac{n+m}{2}}w+b_{24}x^{\frac{m}{2}}z =0, 
	$
	where $b_i \in \mathbb{Q}, 1 \leq i \leq 24;$  $ b_4=0$ if $n< m , b_6=0$ if $l<2m+1,b_7=0$ if   $2m+1>n+l,b_8=b_{23}=0$ if $l<2m+1 , b_{10}=b_{20}=0$ if $l<2n+m+1, b_{11}=b_{24}=0$ if $l<n+2m+1,$ and $b_{12}=0$ if $l<2n+2m+1.$ Thus the cohomology ring $H^*(X_G)$ is given by	$$\mathbb{Q}[x,y,w,z]/<x^{\frac{n+m+l+1}{2}}, I_j,
	x^{\frac{l-m-n+1}{2}}z,x^{\frac{l+n-m+1}{2}}w,x^{\frac{l+m-n+1}{2}}y>_{1\leq j \leq 6},$$
	where deg $x=2,$ deg $y=n$, deg $w=m,$ and deg $z=n+m.$ This realizes possibility (1) when $j'=n+m$ with $c_0=c_1=1$ and $c_2=c_3=0.$\\ 
	Now, let $d_{m+l-n+1}(1\otimes ac)=0.$ Then  we must  have $d_{m+l+1}(1 \otimes bc)=c_1t^{\frac{m+l+1}{2}}\otimes 1$ and $d_{m+l+1}(1 \otimes abc)=c_2t^{\frac{m+l+1}{2}}\otimes a.$ Thus $E_{m+l+2}^{*,*}=E_{\infty}^{*,*}.$ For $n<m<l,$ we have  $E_{\infty}^{p,q}\cong \mathbb{Q},$  $ 0 \leq p(even) < m+l,q=0,n; 0 \leq p(even) < n+l-m,q=m$ and $ 0 \leq p(even) < l-m-n,q=n+m,$ and zero otherwise. For $n=m<l,$ we have $E_{\infty}^{p,q}\cong \mathbb{Q},$  $ 0 \leq p(even) < n+l,q=0; 0 \leq p(even) < l-2n,q=2n;l < p(even) < n+l,q=n; E_{\infty}^{p,q}\cong \mathbb{Q} \oplus \mathbb{Q},$  $ 0 \leq p(even) < l,q=n,$ and zero otherwise. Note that the cohomology groups are same as above. The total complex is Tot $ E_{\infty}^{*,*}=\mathbb{Q}[x,u,v,s]/I,$ where $I$ is an ideal given by $$<x^{\frac{m+l+1}{2}},u^2+\gamma_1v+\gamma_2s,v^2+\gamma_3s,s^2,uv+\gamma_4s,us,vs,x^{\frac{l+n-m+1}{2}}v,x^{\frac{l-m-n+1}{2}}s>, $$ where deg $x=2,$ deg $u= n,$ deg $v= m~\&$   deg $s= n+m,$ and ${\gamma_i \in \mathbb{Q}}, 1\leq i \leq 4;$ $\gamma_1=0$ if $m \neq 2n,$ and $\gamma_2=\gamma_3=0$ if $n < m.$  The ideals ${I_j}'s, 1 \leq j \leq 6$ are also same as above with conditions: $b_{4}=0$ if $n<m,$ $b_{7}=0$ if $2m+1>n+l,$ $b_{8}=b_{23}=0$ if $l<2m+1,$ $b_{9}=0$ if $l<2n+m+1,$ $b_{11}=0$ if $l<n+2m+1,$ $b_{12}=0$ if $l<2n+2m+1,$ $b_{20}=0$ if $l<2n+m+1,$ and $b_{24}=0$ if $l<n+2m+1.$ Thus the cohomology ring $H^*(X_G)$ is given by $$\mathbb{Q}[x,y,w,z]/<x^{\frac{m+l+1}{2}}, I_j,
	x^{\frac{l-m-n+1}{2}}z,x^{\frac{l+n-m+1}{2}}w>_{1\leq j \leq 6},$$
	where deg $x=2,$ deg $y=n$, deg $w=m,$ and  deg $z=n+m.$ This realizes possibility (1) when $j'=m.$ \\
	Now, assume that  $d_{n+l-m+1}$ is trivial. We have either $d_{l+1}(1\otimes ac)=0$ or  $d_{l+1}(1 \otimes ac)= c_1t^{\frac{l+1}{2}}\otimes a.$\\ Let $d_{l+1}(1 \otimes ac)= c_1t^{\frac{l+1}{2}}\otimes a.$ Then $d_{l+1}(1 \otimes bc)= c_2t^{\frac{l+1}{2}}\otimes b,$ and we must have    $d_{n+m+l+1}(1 \otimes abc)= t^{\frac{n+m+l+1}{2}}\otimes 1.$ Thus $E_{n+m+l+2}^{*,*}=E_{\infty}^{*,*}.$ For $n<m<l,$ we have  $E_{\infty}^{p,q}\cong  \mathbb{Q},$  $ 0 \leq p(even) < n+m+l,q=0; 0 \leq p(even) < l,q=n,m$ and $ 0 \leq p(even) < l-m-n,q=n+m,$ and zero otherwise. For $n=m<l,$ we have $E_{\infty}^{p,q}\cong \mathbb{Q},$ $ 0 \leq p(even) < 2n+l,q=0; 0 \leq p(even) < l-2n,q=2n; E_{\infty}^{p,q}\cong \mathbb{Q} \oplus \mathbb{Q},$ $ 0 \leq p(even) < l,q=n,$ and zero otherwise. The cohomology groups of $X_G$ are same as above. The total complex is Tot$E_{\infty}^{*,*}=\mathbb{Q}[x,u,v,s]/I,$ where $I$ is an ideal given by\begin{center}
		$<x^{\frac{n+m+l+1}{2}},u^2+\gamma_1v+\gamma_2s,v^2+\gamma_3s,s^2,uv+\gamma_4s,us,vs,x^{\frac{l+1}{2}}u,x^{\frac{l+1}{2}}v,x^{\frac{l-m-n+1}{2}}s>,$	
	\end{center}  where  deg $x=2,$ deg $u= n,$ deg $v= m~\&$   deg $s= n+m,$ and ${\gamma_i \in \mathbb{Q}}, 1\leq i \leq 4;$ $\gamma_1=0$ if $m \neq 2n,$ and $\gamma_2=\gamma_3=0$ if $n<m.$  The ideals ${I_j}'s; 1 \leq j \leq 6$ are same as above with  conditions: $b_4=0 $ if $n<m, b_8=0$ if $l<2m+1,b_{10}=b_{24}=0$ if $l<n+2m+1,$ $b_{11}=b_{20}=0$ if $l<2n+m+1,$ $b_{12}=0 $ if $l<2n+2m+1,$  and $b_{22}=0$ if $l<2m+1.$  So, the cohomology ring $H^*(X_G)$ is given by $$\mathbb{Q}[x,y,w.z]/<x^{\frac{n+m+l+1}{2}}, I_j,
	x^{\frac{l-m-n+1}{2}}z,x^{\frac{l+1}{2}}y,x^{\frac{l+1}{2}}w>_{1\leq j \leq 6}$$
	where deg $x=2,$ deg $y=n$, deg $w=m,$ and deg $z=n+m.$ This realizes possibility (1) when $j'=n+m$ with $c_0=c_1=0$ and $c_2=c_3=1.$\\
	Now, let $d_{l+1}(1\otimes ac)=0.$ Then we must have  $d_{m+l-n+1}(1\otimes bc)=c_1 t^{\frac{m+l-n+1}{2}}\otimes a.$ Consequently, we get   $d_{n+l+1}(1 \otimes ac)=c
	_2t^{\frac{n+l+1}{2}}\otimes 1$ and  $d_{n+l+1}(1 \otimes abc)=c_2t^{\frac{n+l+1}{2}}\otimes b.$ Thus $E_{n+l+2}^{*,*}=E_{\infty}^{*,*}.$ For $n\leq m<l,$ we have  $E_{\infty}^{p,q}\cong \mathbb{Q},$ $ 0 \leq p(even) < n+l,q=0,m; 0 \leq p(even) < m+l-n,q=n;$ $ 0 \leq p(even) < l-m-n,q=n+m,$ and zero otherwise. The cohomology groups of $X_G$ are same as above. The total complex Tot$E_{\infty}^{*,*}$ is given by $$\mathbb{Q}[x,u,v,s]/<x^{\frac{n+l+1}{2}},u^2+\gamma_1v+\gamma_2s,v^2+\gamma_3s,s^2,uv+\gamma_4s,us,vs,x^{\frac{l+m-n+1}{2}}u,x^{\frac{l-m-n+1}{2}}s>$$  where deg $x=2,$ deg $u= n,$ deg $v= m,$   deg $s= n+m,$ and ${\gamma_i \in \mathbb{Q}}, 1\leq i \leq 4;$ $\gamma_1=0$ if $m \neq 2n,$  and $\gamma_2=\gamma_3=0$ if $m \neq n.$ The ideals ${I_j}'s, 1 \leq j \leq 6$ are same as above  with  conditions: $b_4=0 $ if $n<m, b_5=0$ if $n+l<2m+1,b_{8}=b_{21}=0$ if $l<2m+1,b_{9}=b_{24}=0$ if $l<n+2m+1,$ $b_{10}=0$ if $l<2n+m+1,$  $b_{12}=0$ if $l<2n+2m+1,$ and $b_{20}=0$ if $l<2n+m+1.$ Thus the cohomology ring $H^*(X_G)$ is given by $$\mathbb{Q}[x,y,w,z]/<x^{\frac{n+l+1}{2}}, I_j,
	x^{\frac{l-m-n+1}{2}}z,x^{\frac{m+l-n+1}{2}}y,>_{1\leq j \leq 6}$$
	where deg $x=2,$ deg $y=n$, deg $w=m,$ deg $z=n+m.$ This realizes possibility (1) when $j'=n.$\\
		{\bf Case (ii):} $r_3=l-m+1,$ where $l-m$ odd.\\ Clearly,  $n \leq m <l.$ We have   $d_{l-m+1}(1\otimes c)= c_0t^{\frac{l-m+1}{2}} \otimes b $ and $d_{l-m+1}(1\otimes ac)= c_0t^{\frac{l-m+1}{2}} \otimes ab. $ Now, we have either $d_{m+l-n+1}(1\otimes bc)=0$ or $d_{m+l-n+1}(1 \otimes bc)= c_1t^{\frac{m+l-n+1}{2}} \otimes a.$\\ First, let $d_{m+l-n+1}(1 \otimes bc)= c_1t^{\frac{m+l-n+1}{2}} \otimes a.$  It is clear that  $ d_{n+m+l+1}(1 \otimes abc)= c_2t^{\frac{n+m+l+1}{2}}\otimes 1 .$ Thus $E_{n+m+l+2}^{*,*}=E_{\infty}^{*,*}.$ For $n<m<l,$ we have  $E_{\infty}^{p,q}\cong\mathbb{Q},$  $ 0 \leq p \leq n+m+l-1,q=0; 0 \leq p \leq m+l-n-1,q=n;$  $ 0 \leq p \leq l-m-1,q=m,n+m,$ and zero otherwise. For $n=m<l,$ we have $E_{\infty}^{p,q}\cong \mathbb{Q},$  $ 0 \leq p \leq 2n+l-1,q=0; 0 \leq p \leq l-n-1,q=2n,l-n-1 < p \leq l-1,q=n;  E_{\infty}^{p,q}\cong \mathbb{Q} \oplus \mathbb{Q},$  $ 0 \leq p \leq l-n-1,q=n,$ and zero otherwise.\\  For $n<m<l<n+m,$ the cohomology groups of $X_G$ are given by 
	\[
	H^k(X_G)=
	\begin{cases}
		\mathbb{Q}  & 0 \leq k <n ; m+l \leq k \leq m+n+l,\\
		\mathbb{Q} \oplus \mathbb{Q} & n \leq k < m;l-1<k<n+m-1;n+l-1 < k \leq m+l-1, \\
		\mathbb{Q} \oplus \mathbb{Q} \oplus \mathbb{Q} & m+j \leq k \leq l+j-1,j=0,n,\\
		0 & \mbox{otherwise,}
	\end{cases}
	\] and, for $n=m<l<2n,$ we have 
	\[
	H^k(X_G)=
	\begin{cases}
		\mathbb{Q}  & 0 \leq k < n; n+l \leq k \leq 2n+l,\\
		\mathbb{Q} \oplus \mathbb{Q} & l\leq k<2n, \\
		\mathbb{Q} \oplus \mathbb{Q} \oplus \mathbb{Q} & n+j \leq k \leq l+j-1,j=0,n,\\
		0 & \mbox{otherwise.}
	\end{cases}
	\]  For $l \geq n+m,$ the cohomology groups are similar to the case (i). 
	The total complex is Tot$E_{\infty}^{*,*} = \mathbb{Q}[x,u,v,s]/I ,$ where $I$ is an ideal given by \begin{center}
		$	<x^{\frac{n+m+l+1}{2}},u^2+\gamma_1v+\gamma_2s,v^2+\gamma_3s,s^2,uv+\gamma_4s,us,vs,x^{\frac{l+m-n+1}{2}}u,x^{\frac{l-m+1}{2}}v,x^{\frac{l-m+1}{2}}s>,$
	\end{center}  where  deg $x=2,$ deg $u= n,$ deg $v= m,$   deg $s= n+m,$ and ${\gamma_i \in \mathbb{Q}},$ $ 1\leq i \leq 4;$ $\gamma_1=0$ if $m \neq 2n,$ and $\gamma_2=\gamma_3=0$ if $n<m.$ The ideals ${I_j}'s, 1 \leq j \leq 6$ are same as in the case (i)  with  conditions: $b_3=b_{18}=0$ if $l<2n+1,$ $b_4=0 $ if $ n<m, b_7=b_{24}=0$ if $l<2m+1,b_{8}=0$ if $n+l<2m+1,b_{9}=b_{15}=b_{20}=b_{22}=0$ if $l<n+m-1,$ $b_{10}=b_{19}=0$ if $l<2n+m+1, b_{11}=0$ if $l<2n+2m+1,$ and $b_{12}=b_{23}=0$ if $l<n+2m+1.$ Thus the cohomology ring $H^*(X_G)$ is given by $$\mathbb{Q}[x,y,w.z]/<x^{\frac{n+m+l+1}{2}}, I_j,
	x^{\frac{m+l-n+1}{2}}y,x^{\frac{l-m+1}{2}}w,x^{\frac{l-m+1}{2}}z>_{1\leq j \leq 6}$$
	where deg $x=2,$ deg $y=n$, deg $w=m,$ and deg $z=n+m.$ This realizes possibility (2) when $j'=n.$\\
	Now, let  $d_{m+l-n+1}(1\otimes bc)=0.$ Then  we must have  $d_{m+l+1}(1 \otimes bc)=c_1t^{\frac{m+l+1}{2}}\otimes 1$ and $d_{m+l+1}(1 \otimes abc)=c_2t^{\frac{m+l+1}{2}}\otimes a.$ Thus $E_{m+l+2}^{*,*}=E_{\infty}^{*,*}.$ For $n<m<l,$ we have  $E_{\infty}^{p,q}\cong \mathbb{Q},$ if $ 0 \leq p \leq m+l-1,q=0,n; 0 \leq p \leq l-m-1,q=m,n+m,$ and  zero otherwise. For $n=m<l,$ we have $E_{\infty}^{p,q}\cong \mathbb{Q},$ $ 0 \leq p \leq n+l-1,q=0; l-n-1 < p \leq n+l-1,q=n;0 \leq p \leq l-n-1,q=n+m;  E_{\infty}^{p,q}\cong \mathbb{Q} \oplus \mathbb{Q},$ $ 0 \leq p \leq l-n-1,q=n,$ and zero otherwise. For $l<n+m,$ the cohomology groups are same as in the case (ii) when  $d_{m+l-n+1}$ is nontrivial, and  for $l \geq n+m,$  the cohomology groups are similar to the case (i). 
	Thus the total complex is Tot $E_{\infty}^{*,*}=\mathbb{Q}[x,u,v,s]/I,$ where  $I$ is an ideal given by  \begin{center}
		$<x^{\frac{m+l+1}{2}},u^2+\gamma_1v+\gamma_2s,v^2+\gamma_3s,s^2,uv+\gamma_4s,us,vs,x^{\frac{l-m+1}{2}}v,x^{\frac{l-m+1}{2}}s>,$
	\end{center} where   deg $x=2,$ deg $u= n,$ deg $v= m,$   deg $s= n+m,$ and ${\gamma_i \in \mathbb{Q}},$ $ 1\leq i \leq 4;$ $\gamma_1=0$ if $m \neq 2n,$ and $\gamma_2=\gamma_3=0$ if $n<m.$ The ideals ${I_j}'s, 1 \leq j \leq 6$ are same as in the case (i)  with  conditions: $b_3=b_{17}=0$ if $l<2n+1,b_4=0$ if $n<m, b_7=b_{24}=0$ if $l<2m+1,b_8=0$ if $n+l<2m+1,b_9=b_{19}=0$ if $l<2n+m+1, b_{10}=b_{15}=b_{20}=b_{21}=0$ if $l<m+n+1,b_{11}=0$ if $l<2n+2m+1,$ $b_{12}=b_{23}=0$ if $l<2m+n+1,$ and $b_{17}=0$ if $l<2n+1.$  Thus the cohomology ring $H^*(X_G)$ is given by $$\mathbb{Q}[x,y,w,z]/<x^{\frac{m+l+1}{2}}, I_j,
	x^{\frac{l-m+1}{2}}w,x^{\frac{l-m+1}{2}}z>_{1\leq j \leq 6},$$
	where deg $x=2,$ deg $y=n,$ deg $w=m$ and  deg $z=n+m.$ This realizes possibility (2) when $j'=0.$\\
		{\bf Case (iii):}  $r_3=l-n+1,$ where $l-n$ odd. \\Clearly, $n<l.$ We have  $d_{l-n+1}(1\otimes c)= c_0t^{\frac{l-n+1}{2}} \otimes a $ and $d_{l-n+1}(1\otimes bc)= c_0t^{\frac{l-n+1}{2}} \otimes ab.$ We have either  $d_{n+l-m+1}(1\otimes ac)=0$ or $d_{n+l-m+1}(1 \otimes ac)=c_1t^{\frac{n+l-m+1}{2}}\otimes b.$\\  First, let   $d_{n+l-m+1}(1 \otimes ac)=c_1t^{n+l-m+1}\otimes b.$ Then we  must have $d_{n+m+l+1}(1 \otimes abc)= c_2t^{n+m+l+1}\otimes 1.$ Thus $E_{n+m+l+2}^{*,*}=E_{\infty}^{*,*}.$ For $n<m\leq l,$ we have  $E_{\infty}^{p,q}\cong \mathbb{Q}$ if $ 0 \leq p \leq n+m+l-1,q=0; 0 \leq p \leq l-n-1,q=n,n+m; 0 \leq p \leq n+l-m-1,q=m,$ and  zero otherwise. For $n=m<l,$ we have $E_{\infty}^{p,q}\cong \mathbb{Q}$ if $ 0 \leq p \leq 2n+l-1,q=0; l-n-1 < p \leq l-1,q=n;0 \leq p \leq l-n-1,q=2n; E_{\infty}^{p,q}\cong\mathbb{Q} \oplus \mathbb{Q}$ if $ 0 \leq p \leq l-n-1,q=n,$ and zero otherwise. For $l \geq n+m,$ the cohomology groups are similar to the case (i),  and  for $l<n+m,$ the cohomology groups are same as in the  case (ii). Thus the total complex is Tot$E_{\infty}^{*,*}$ = $\mathbb{Q}[x,u,v,s]/I,$ where $I$ is an ideal given by \begin{center}
		$<x^{\frac{n+m+l+1}{2}},u^2+\gamma_1v+\gamma_2s,v^2+\gamma_3s,s^2,uv+\gamma_4s,us,vs,x^{\frac{l-n+1}{2}}u,x^{\frac{n+l-m+1}{2}}v,x^{\frac{l-n+1}{2}}s>,$
	\end{center}  where deg $x=2,$ deg $u= n,$ deg $v= m,$  deg $s= n+m,$ and ${\gamma_i \in \mathbb{Q}},$ $ 1\leq i \leq 4;$ $\gamma_1=0$ if $m \neq 2n,$ and $\gamma_2=\gamma_3=0$ if $n < m.$  Clearly, the ideals $I_j, 1 \leq j \leq 6$ are same as in the case (i) with  conditions: $b_2=b_{20}=0$ if $l<2n+1,$ $b_4=0 $ if $ n<m, b_6=b_{23}=0$ if $l<2m+1,b_{7}=0$ if $n+l<2m+1,b_{9}=b_{14}=b_{19}=b_{24}=0$ if $l<n+m-1,b_{10}=0$ if $l<2n+2m+1,$ $b_{11}=b_{22}=0$ if $l<2m+n+1$ and $b_{12}=b_{18}=0$ if $l<2n+m+1.$    Thus the cohomology ring $H^*(X_G)$ is given by $$\mathbb{Q}[x,y,w.z]/<x^{\frac{n+m+l+1}{2}}, I_j,
	x^{\frac{l-n+1}{2}}y,x^{\frac{n+l-m+1}{2}}w,x^{\frac{l-n+1}{2}}z>_{1\leq j \leq 6}$$
	where deg $x=2,$ deg $y=n$, deg $w=m,$ and deg $z=n+m.$ This realizes possibility (3) when $j'=m.$\\
	Now, let   $d_{n+l-m+1}(1\otimes ac)=0.$ Then we must have  $d_{n+l+1}(1 \otimes ac)=c_1t^{m+l+1}\otimes 1$ and $d_{n+l+1}(1 \otimes abc)=c_1t^{m+l+1}\otimes b.$ Thus $E_{n+l+2}^{*,*}=E_{\infty}^{*,*}.$ For $n<m\leq l,$ we have  $E_{\infty}^{p,q}\cong \mathbb{Q},$ if $ 0 \leq p \leq n+l-1,q=0,m; 0 \leq p \leq l-n-1,q=n,n+m,$ and  zero otherwise. For $n=m<l,$ we have $E_{\infty}^{p,q}\cong \mathbb{Z}_2,$ if $ 0 \leq p \leq n+l-1,q=0; l-n-1 < p \leq n+l-1,q=n;0 \leq p \leq l-n-1,q=2n;  E_{\infty}^{p,q}\cong \mathbb{Q} \oplus \mathbb{Q}$ if $ 0 \leq p \leq l-n-1,q=n,$ and zero otherwise. For $l \geq n+m,$ the cohomology groups are similar to the case (i), and for $l<n+m,$ the cohomology groups are same as in the case (ii). Thus the total complex is Tot $E_{\infty}^{*,*}$ = $\mathbb{Q}[x,u,v,s]/I,$ where  $I$ is an ideal given by \begin{center}
		$<x^{\frac{n+l+1}{2}},u^2+\gamma_1v+\gamma_2s,v^2+\gamma_3s,s^2,uv+\gamma_4s,us,vs,x^{\frac{l-n+1}{2}}u,x^{\frac{l-n+1}{2}}s>,$
	\end{center} where  deg $x=2,$ deg $u= n,$ deg $v= m,$  deg $s= n+m,$ and ${\gamma_i \in \mathbb{Q}},$ $ 1\leq i \leq 4;$ $\gamma_1=0$ if $m \neq 2n,$ and $\gamma_2=\gamma_3=0$ if $ n<m.$  The ideals ${I_j}'s; 1 \leq j \leq 6$ are same as in the case (i) with  conditions: $b_2=b_{20}=0$ if $l<2n+1,$ $b_4=0 $ if $ n<m, $ $b_{5}=0$ if $n+l<2m+1,$ $b_6=b_{21}=0$ if $l<2m+1,$ $b_{9}=b_{12}=0$ if $l<2m+n+1,$  $b_{10}=0$ if  $l<2n+2m+1,b_{11}=b_{14}=b_{17}=b_{24}=0$ if $l<n+m+1$ and $b_{18}=b_{22}=0$ if $l<2n+m+1.$   Thus the cohomology ring $H^*(X_G)$ is given by $$\mathbb{Q}[x,y,w,z]/<x^{\frac{n+l+1}{2}}, I_j,
	x^{\frac{l-n+1}{2}}y,x^{\frac{l-n+1}{2}}z>_{1\leq j \leq 6},$$
	where deg $x=2,$ deg $y=n,$ deg $w=m$ and deg $z=n+m.$ This realizes possibility (3) when $j'=0.$\\
	{\bf Case (iv):}  $r_3=l+1,$ where $l$ is odd.\\ We have  $d_{l+1}(1\otimes c)= c_0t^{\frac{l+1}{2}} \otimes 1. $  Consequently,   $d_{l+1}(1\otimes ac)= c_0t^{\frac{l+1}{2}} \otimes a,$ $d_{l+1}(1\otimes bc)= c_0t^{\frac{l+1}{2}} \otimes b$ and  $d_{l+1}(1\otimes abc)= c_0t^{\frac{l+1}{2}} \otimes ab.$ Thus $E_{l+2}^{*,*}=E_{\infty}^{*,*}.$ For $n<m\leq l ,$ we have  $E_{\infty}^{p,q}\cong \mathbb{Q},$ where $ 0 \leq p \leq l-1,q=0,n,m,n+m,$ and  zero otherwise. For $n=m\leq l,$ we have $E_{\infty}^{p,q}\cong \mathbb{Q},$  $ 0 \leq p \leq l-1,q=0,2n;  E_{\infty}^{p,q}\cong \mathbb{Q} \oplus \mathbb{Q},$  $ 0 \leq p \leq l-1,q=n,$ and  zero otherwise. For $l \geq n+m,$ the cohomology groups are similar to the case (i), and  for $l<n+m,$ the cohomology groups are same as in the case (ii). The total complex is Tot $E_{\infty}^{*,*}$ = $\mathbb{Q}[x,u,v,s]/I,$ where  $I$ is an ideal given by \begin{center}
		$<x^{\frac{l+1}{2}},u^2+\gamma_1v+\gamma_2s,v^2+\gamma_3s,s^2,uv+\gamma_4s,us,vs>,$
	\end{center} where  deg $x=2,$ deg $u= n,$ deg $v= m,$  deg $s= n+m,$ and ${\gamma_i \in \mathbb{Q}},$ $ 1\leq i \leq 4;$ $\gamma_1=0$ if $m \neq 2n,$ and $\gamma_2=\gamma_3=0$ if $ n<m.$ The ideals ${I_j}'s; 1 \leq j \leq 6$ are same as in case (i) with conditions:
	$b_1=b_{19}=0$ if $l<2n+1,$ $b_4=0 $ if $ n<m, b_5=b_{22}=0$ if $l<2m+1,b_{6}=0$ if $n+l<2m+1, b_{9}=0$ if $l<2n+2m+1,$ $b_{10}=b_{21}=0$ if $l<2m+n+1,$ $b_{11}=b_{17}=0$ if $l<2n+m+1,$ and $b_{12}=b_{13}=b_{18}=b_{23}=0$ if $l<n+m+1.$ Thus the cohomology ring $H^*(X_G)$ is given by $$\mathbb{Q}[x,y,w.z]/<x^{\frac{l+1}{2}}, I_j>_{1\leq j \leq 6},$$
	where deg $x=2,$ deg $y=n$, deg $w=m$ and deg $z=n+m.$ This realizes possibility (4).	
		\end{proof}
		\begin{example}
		An example of case (1) of Theorem \ref{thm 3.6} can be realized by considering  diagonal action of $G=\mathbb{S}^1$ on $\mathbb{S}^n \times \mathbb{S}^m \times \mathbb{S}^l,$ where $\mathbb{S}^1$ acts freely on $\mathbb{S}^n$ and trivially on both $\mathbb{S}^m$ and $\mathbb{S}^l.$  Then $X/G \sim_\mathbb{Q} \mathbb{CP}^{\frac{n-1}{2}} \times \mathbb{S}^m \times \mathbb{S}^l.$ This realizes case (1) by taking $a_4=1$ and $a_i =0$ for $i\neq 4.$ Similarly, case (2) of Theorem  \ref{thm 3.7}, with $a_6=1$ and $a_i=0$ for $i\neq 6,$  and case (4) of Theorem \ref{thm 3.8}, with $a_{16}=1$ and $a_i=0$ for $i\neq 16,$  can be realized.
	\end{example}
		\section{Applications}
	In this section,  we derive the Borsuk-Ulam type results for free $G=\mathbb{S}^1$ actions on finitistic space $X \sim_\mathbb{Q} \mathbb{S}^n \times \mathbb{S}^m \times \mathbb{S}^l ,1\leq n \leq m \leq l.$ We determine the nonexistence of $G$-equivariant maps between  $ X$ and $ \mathbb{S}^{2k+1},$ where $\mathbb{S}^{2k+1}$ equipped with  the standard free $G$ action namely, componentwise multiplication.\\
	\indent Recall that \cite{floyd} the index (respectively, co-index) of a $G$-space $X$ is the greatest integer $k$ (respectively, the lowest integer $k$)  such that there exists a $G$-equivariant map $\mathbb{S}^{2k+1} \rightarrow X$ (respectively, $ X \rightarrow  \mathbb{S}^{2k+1}$).\\
	\indent	By  Theorems proved in Section 3, we get the largest integer  $s=\frac{r-1}{2}$ for which $w^s \neq 0,$ where $w \in H^{2}(X/G)$ is the  characteristic class of the principle $G$-bundle $G \hookrightarrow X \rightarrow X/G,$ and  $r$ is one of the following: $ n,m,l,n+m,2n+l,n+l,m+l$ or $n+m+l,$ We know that index$(X)\leq s$ \cite{floyd}. Thus we have the following Result:
		\begin{theorem}
		Let $G=\mathbb{S}^1$ act freely on a  finitistic space $X \sim_\mathbb{Q}\mathbb{S}^n \times \mathbb{S}^m \times \mathbb{S}^l,$  $1\leq n \leq m \leq l.$ Then there does not exist $G$-equivarient map from $\mathbb{S}^{2k+1} \rightarrow X,$ for $k > \frac{r-1}{2},$ where $r$ is one of the following: $ n,m,l,n+m,2n+l,n+l,m+l$ or $n+m+l.$
	\end{theorem}
	Recall that the Volovikov's index $i(X)$ is the smallest integer $r\geq 2$ such that $d_r: E_r^{k-r,r-1} \rightarrow E_r^{k,0}$ is nontrivial for some $k,$ in the Leray-Serre spectral sequence of the Borel fibration  $ X \stackrel{i} \hookrightarrow X_G \stackrel{\pi} \rightarrow B_G$ \cite{yu}. Again, 	by  Theorems proved in Section 3, we get $i(X)$ is one of the following: $ n+1, m+1,l+1,m-n+1,l-m+1,l-n+1$ or $l-m-n+1.$
	By taking $Y=\mathbb{S}^{k}$ in Theorem 1.1 \cite{co}, we have

\begin{theorem}
	Let $G=\mathbb{S}^1$ act freely on a  finitistic space $X \sim_\mathbb{Q} \mathbb{S}^n \times \mathbb{S}^m \times  \mathbb{S}^l, 1 \leq n\leq m \leq l .$ Then there is no $G$-equivariant map $f: X \rightarrow \mathbb{S}^{2k+1}$  if $2k+1< i(X)-1,k \geq 1,$ where $i(X)$ is one of the following: $ n+1, m+1,l+1,m-n+1,l-m+1,l-n+1$ or $l-m-n+1.$	
\end{theorem}

	\end{document}